\documentclass[]{article}

\usepackage{amsmath}
\usepackage{amsfonts}
\usepackage{amssymb}
\usepackage[amsmath,thmmarks,standard]{ntheorem}
\usepackage{tikz}
\usepackage{authblk}

\usepackage[utf8x]{inputenc}
\usepackage[hebrew,english,bidi=default]{babel}

\usepackage[a4paper, hmargin=3cm, vmargin={4cm, 4cm}]{geometry}

\newcommand{\RR}{\mathbb{R}}
\newcommand{\DD}{\mathbb{D}}
\newcommand{\CC}{\mathbb{C}}
\newcommand{\sing}{\operatorname{Sing}}

\newcommand{\sph}{\mathbb{S}}

\newcommand{\ddn}{\cdot d\hat{\mathbf{n}}}

\renewtheorem{theorem}{Theorem}
\numberwithin{theorem}{subsection}
\newtheorem{claim}[theorem]{Claim}
\renewtheorem{lemma}[theorem]{Lemma}

\renewtheorem{definition}[theorem]{Definition}

\renewtheorem{corollary}[theorem]{Corollary}

\newcommand{\myenglishtitle}{Constructing Riemannian metrics with prescribed nodal sets for Laplacian eigenfunctions}
\newcommand{\myenglishabstract}{Let $C$ be a configuration of $n$ non-intersecting and smooth ovals in $\sph^2$. We show that there is a Riemannian metric $g$ over $\sph^2$ with a Laplacian eigenfunction whose zero set is $C$, and the corresponding eigenvalue is the $k$-th eigenvalue for $n\leq k \leq \alpha_1 n$. We also have that $\lambda\operatorname{Vol}_g\left(\sph^2\right) = \Theta(n)$. This extends a result by Lisi.
	
	Additionally, assuming $C$ can be drawn as a subgraph of the $m\times m$ grid graph, we show that there is an infinitesimal perturbation of the round metric on $\sph^2$ and a corresponding Laplacian eigenfunction $f$ with eigenvalue $\Theta(m^2)$ such that the zero set of $f$ is topologically equivalent to $C$.}

\newcommand{\myenglishacknowledgements}{I thank Mikhail Sodin for his guidance, discussions, and his many comments on various drafts of this work.
	Additionally I thank Alexander Logunov for discussions and suggesting a nice question which became Theorem \ref{thm:main-perturb-result}.
	I also thank Daniel Peralta-Salas for telling about Lisi's work.}

\title{\myenglishtitle}
\author{Yoav Krauz}
\affil{School of Mathematical Sciences, Tel Aviv University, Tel Aviv 69978, Israel. Email: { \tt yoav.krauz@gmail.com}}

\begin{document}
\maketitle
\begin{abstract}
	\myenglishabstract
\end{abstract}
\section{Introduction}
Given a Riemannian surface $(M,g)$, one may look at eigenfunctions $f$ of the minus $g$-Laplacian operator on $M$, and study various properties of such eigenfunctions $f$. In this work, we study the following problem: given an embedding $C$ of topological circles into $M$, how to choose a Riemannian metric $g$ on $M$ such that $C$ would be the zero set of some eigenfunction $f$ of the minus $g$-Laplacian operator on $M$?

Denote by $\lambda_n$ the $n$-th eigenvalue of the minus $g$-Laplacian operator on $M$.
A theorem of Courant \cite{Courant} says that for an eigenfunction $f$ which corresponds to $\lambda_n$, the zero set of $f$ divides $M$ to at most $n$ regions, which are called nodal domains.

Recall that in the case of $M=\sph^2$ with the regular metric, the eigenvalues are $\lambda=n\left(n+1\right)$ with multiplicity $2n+1$.
In this case, Lewy \cite{Lewy} showed that for every eigenvalue $\lambda=n\left(n+1\right)>0$ there is an eigenfunction whose zero set is either a single curve (when $n$ is odd) or the disjoint union of two closed curves (when $n$ is even).

Continuing with the case of $\sph^2$ with the regular metric, Eremenko, Jakobson and Nadirashvili \cite{NodalSetEremenko} showed that for any set of $n$ disjoint closed curves on $\sph^2$, whose union $C$ is invariant with respect to the antipodal map, there is an eigenfunction $f$ corresponding to $\lambda=n\left(n+1\right)$ such that the zero set of $f$ is topologically equivalent to $C$.

Lisi \cite{Lisi}, using the uniformization theorem, showed that, given a closed connected surface $M$ and a collection of smooth closed curves $C\subset M$ dividing $M$ into two regions, there is a Riemannian metric $g$ on $M$ and an eigenfunction $f$ of the minus $g$-Laplacian such that the zero set of $f$ is $C$.

Canzani and Sarnak \cite{CanzaniSarnak} showed that a bounded connected component of the zero set of a solution to $\Delta u + u = 0$ in $\RR^n$ can have the topology of an arbitrary compact smooth manifold; this is also considered by Enciso and Peralta-Salas \cite{ENCISO2013204} (remark A2).
Analogously, Enciso and Peralta-Salas \cite{ENCISO2013204} showed that for any smooth embedded hypersurface $L$ of $\RR^n$ such that $L$ has no compact connected component and that $L$ is a nonsingular real algebraic hypersurface (this condition can be relaxed), there is a smooth diffeomorphism $\Phi$ of $\RR^n$ such that $\Phi(L)$ is a union of connected components of a level set of a function $u$ satisfying $\Delta u - u = 0$ in $\RR^n$.

Enciso and Peralta-Salas \cite{EncistoPeraltaSalas} showed that, given a closed manifold $M$ with $\dim M\geq 3$ and a closed connected oriented hypersurface $S\subset M$ which divides $M$, there is a Riemannian metric $g$ on $M$ such that $S$ is the zero-set of the first nonconstant eigenfunction $f$ of the minus $g$-Laplacian.

\section{Results}

\begin{figure}
	\centering
	\begin{tikzpicture}[x=0.75pt,y=0.75pt,yscale=-1,xscale=1]
		\def\gridx{190}
		\def\gridy{-40}
		\def\boxwidth{30}
		\def\boxheight{30}
		\def\gridsize{5}
		\def\gridlinewidth{1.5}
		
		\foreach \i in {0,...,\gridsize} {
			\draw [line width=\gridlinewidth] (\gridx,\gridy+\i*\boxheight) -- (\gridx+\gridsize*\boxwidth,\gridy+\i*\boxheight);
		}
		\foreach \j in {0,...,\gridsize} {
			\draw [line width=\gridlinewidth] (\gridx+\j*\boxwidth,\gridy) -- (\gridx+\j*\boxwidth,\gridy+\gridsize*\boxheight);
		}
		
		\def\drawinglinewidth{5}
		
		\newcommand{\lineingrid}[4]{
			\draw [line width=\drawinglinewidth] (\gridx+#1*\boxwidth,\gridy+#2*\boxheight) -- (\gridx+#3*\boxwidth,\gridy+#4*\boxheight);
		}
		
		\lineingrid{0}{0}{3}{0}
		\lineingrid{3}{0}{3}{5}
		\lineingrid{3}{5}{0}{5}
		\lineingrid{0}{5}{0}{0}
		\lineingrid{1}{1}{2}{1}
		\lineingrid{2}{1}{2}{2}
		\lineingrid{2}{2}{1}{2}
		\lineingrid{1}{2}{1}{1}
		\lineingrid{1}{3}{2}{3}
		\lineingrid{2}{3}{2}{4}
		\lineingrid{2}{4}{1}{4}
		\lineingrid{1}{4}{1}{3}
		\lineingrid{4}{0}{5}{0}
		\lineingrid{5}{0}{5}{1}
		\lineingrid{5}{1}{4}{1}
		\lineingrid{4}{1}{4}{0}
		\lineingrid{4}{2}{5}{2}
		\lineingrid{5}{2}{5}{4}
		\lineingrid{5}{4}{4}{4}
		\lineingrid{4}{4}{4}{2}
		
		\draw (0,0) ellipse (70 and 80);
		\draw (-9,-30) ellipse (30 and 20);
		\draw (11,30) ellipse (30 and 20);
		\draw (-20,110) ellipse (25 and 20);
		\draw (40,107) ellipse (25 and 20);
		
	\end{tikzpicture}
	\caption{An example of a configuration $C$ of $5$ ovals and a drawing of $C$ in the $5\times 5$ grid graph}
	\label{fig:drawing-ovals-in-grid}
\end{figure}

Our first result is the following, extending the above-mentioned result of Lisi \cite{Lisi} with asymptotic estimation of the eigenvalue index and the geometry of the metric:
\begin{theorem}\label{thm:ovals}
	For any configuration $C$ of $n$ non-intersecting and smooth ovals on $\sph^2$, there is a Riemannian metric $g$ over $\sph^2$ and a $g$-Laplacian eigenfunction $f:\sph^2\to\RR$ with eigenvalue $-\lambda$ such that the zero set of $f$ is $C$, and $\lambda_{ n}\leq\lambda \leq \lambda_{\alpha_1 n}$, and $\lambda \operatorname{Vol}_g(\sph^2) \in\left[\alpha_2 n, \alpha_3 n\right]$ (here $\alpha_1, \alpha_2, \alpha_3 > 0$ are absolute constants).
	
	Our construction of $g$ also satisfies the property that at each point $p\in\sph^2$, the Gaussian curvature ${\kappa_g(p)}$ of $g$ at $p$ satisfies ${\left|\kappa_g(p) \right|<\alpha_4 \lambda}$ (here $\alpha_4$ is an absolute constant).
\end{theorem}

While it is possible by principle to calculate estimates for the constants $\alpha_i$ used in the statement of theorem \ref{thm:ovals}, in our opinion such calculation would probably be long, technical, and would result in estimates which are unlikely to be sharp.

To prove theorem \ref{thm:ovals} we will prove that any function $f$ on a Riemannian surface with boundary $M$ such that $f$ has no critical points and $f$ is a Laplacian eigenfunction in a neighborhood of $\partial M$, under certain conditions which are also necessary, can be made into an actual Laplacian eigenfunction by changing the metric away from $\partial M$ (see theorem \ref{thm:interpolation-metric} in section \ref{section:interpolation-formulation}).

To state our second result we will formally define when an oval configuration can be drawn in a graph:
\begin{definition} \label{def:can-be-drawn}
	Let $C$ be a configuration of non-intersecting and smooth ovals on $\sph^2$, and let $G$ be a finite planar graph with a fixed embedding into $\sph^2$. Then a \emph{drawing} of $C$ in $G$ is a subgraph $G'$ of $H$ which is $2$-regular (hence each connected component of $H$ is a cycle) such that the set of connected components of $H$, each one viewed as an oval, is topologically equivalent to $C$.
	
	We will say that $C$ \emph{can be drawn} in $G$ if such drawing as above exists.
\end{definition}
See figure \ref{fig:drawing-ovals-in-grid} for an example of a drawing of an oval configuration in a grid graph.

Our second result is the following, which affirmatively answers a question raised by Logunov in a private communication:
\begin{theorem}\label{thm:main-perturb-result}
	There is an absolute constant $C$ that the following holds: let $n\geq 1$ be a natural number, and let $X$ be an oval configuration in $\sph^2$ which can be drawn in the $n\times n$ grid graph. Then there is an infinitesimal perturbation $g$ of the round metric on $\sph^2$ and an eigenfunction $f$ of the minus Laplacian $-\Delta_g$ such that $f$ is a perturbation of a spherical harmonic of degree $Cn$ and the zero set of $f$ is topologically equivalent to $X$.
\end{theorem}

It should be noted that the condition that $X$ can be drawn in the $n\times n$ grid graph is not to restrict the set of configurations for which theorem \ref{thm:main-perturb-result} applies; it just means that $n$ measures the complexity of $X$, differently then counting the ovals in $X$.
Indeed, note that any configuration of $m$ non-intersecting and smooth ovals can be drawn in a $k \times k$ grid graph with $k=O(m)$. However, for some configurations one can do better: for example, the configuration which consists of $m$ ovals where no oval is inside any other oval, can be drawn in a $k\times k$ grid graph with $k=O(\sqrt k)$. On the other hand, for the configuration which is formed by $m$ concentric circles to be drawn in a $k\times k$ grid graph we must have $k=\Omega(m)$.

\section{Acknowledgments}
\myenglishacknowledgements

\section{Notation}\label{section:notation}
Throughout this text, all the surfaces being considered will be smooth and orientable.

A measure $\mu$ on a surface will be called \emph{smooth} if on any local coordinate chart it can be written as $F d\nu$ where $d\nu$ is the Lebesgue measure and $F$ is a strictly-positive smooth function.
On a smooth (orientable) surface, every smooth measure uniquely corresponds to a smooth $2$-form such that the two agree on the area of any positively-oriented cell. Locally, the smooth measure $Fd\nu$ corresponds to the $2$-form $Fdx\wedge dy$ where $x,y$ are the local coordinates.
A smooth Riemannian metric $g$ naturally gives rise to a smooth $2$-form given locally as $\sqrt{\left|g\right|}dxdy$, and therefore also to a smooth measure $\mu$. For this choice of $\mu$ we will say that $\mu$ is \emph{compatible with} $g$ (or that $g$ is compatible with $\mu$).

Given a Riemannian metric $g$ and a smooth function $f$ we denote by $\nabla_g f$ the gradient of $f$ (as a vector field) calculated using the metric $g$. Given a vector field $u$ we denote by $\nabla \cdot u$ the divergence of $u$ (note that it is implicitly depends only on a measure, not on a Riemannian metric); and we denote by $\Delta_g$ the Laplacian operator corresponding to the metric $g$.

Given a smooth surface $M$ with a smooth area measure $\mu$, a vector field $u$, and a piecewise smooth regular curve $\gamma:\left[0,L\right]\to M$ we denote the flux integral of $u$ through $\gamma$ in the following way
\[
\int_\gamma u\cdot d\hat{\mathbf{n}} 
= \int_0^L \mu(u(t)\wedge \dot\gamma(t))dt
,
\]
where $\mu(v_1 \wedge v_2)$ is the local area $2$-form corresponding to $\mu$ applied on vectors $v_1, v_2$.
Note that this does not depends on a Riemannian metric on $M$ except for the area measure $\mu$.
When $\gamma$ is chosen to parameterize the boundary $\partial R$ of a piecewise smooth subset $R\subset M$, we will choose its orientation so that $\int_{\partial M} u\cdot d\hat{\mathbf{n}} := \int_\gamma u\cdot d\hat{\mathbf{n}}$ will be positive when $u$ points to outside $R$.

\section{Common Lemma}
In both of our results we use the following lemma, which allows us to translate the problem of building a Riemannian metric $g$ into building a vector field to act as the gradient $\nabla_g f$ of a given function $f$:

\begin{lemma}\label{lemma:pointwise}
	Let $M$ be a smooth surface endowed with a smooth measure $\mu$. Given a smooth vector field $u$ and a smooth covector field $\omega$ such that $\left\langle \omega,u\right\rangle>0$, there is a unique Riemannian metric $g$ compatible with $\mu$ such that $\omega_i = g_{ij}u^j$. Additionally, $g$ is smooth.
\end{lemma}

Proof of lemma \ref{lemma:pointwise} is standard and given in appendix \ref{section:appendix-proof-of-pointwise-lemma}.

\section{Blocks}
To prove theorem \ref{thm:ovals} we will show the existence of $3$ building blocks that can be joined together to form the construction required in theorem \ref{thm:ovals}.
\begin{definition} \label{def:block}
	A \emph{block} consists of a compact Riemannian surface with boundary $(M,g)$ with a smooth function $f$ on it, with the following properties:
\begin{enumerate}
	\item \label{cond:block-eigenfunc} $f$ is an eigenfunction of the minus $g$-Laplacian with eigenvalue $\lambda=1$;
	\item $f$ does not change sign on $M$;
	\item \label{cond:block-boundary-isometry} For each boundary component $C\subset \partial M$ there is a neighborhood $C\subset U\subset M$ with an isometry $\phi_C:U\to\left[0,\epsilon\right)\times\sph^1$ for some $\epsilon>0$;
	\item \label{cond:block-boundary-func} For each boundary component $C\subset\partial M$ and with the corresponding $U,\phi_C,\epsilon$ as above, we have that the function $f\circ\phi_C^{-1}:\left[0,\epsilon\right)\times\sph^1\to\RR$ is either
	\begin{enumerate}
		\item $f\circ\phi_C^{-1}(x,y) = \sin x$, or
		\item $f\circ\phi_C^{-1}(x,y) = \cos x$.
	\end{enumerate}
	In the first case we will say that $C$ is a \emph{Dirichlet-type} boundary component, and in the second case we will say that $C$ is a \emph{Neumann-type} boundary component.
\end{enumerate}
\end{definition}
The following types of blocks will be called \emph{simple}:
\begin{enumerate}
	\item $M$ is diffeomorphic to $\DD$, the disk in $\RR^2$, with the only boundary component of $\partial M$ being a Dirichlet-type
	\item $M$ is diffeomorphic to $\left[0,1\right]\times\sph^1$, with the boundary component $\left\{0\right\}\times\sph^1$ being Dirichlet-type and the boundary component $\left\{1\right\}\times\sph^1$ being Neumann-type
	\item $M$ is diffeomorphic to a pair-of-pants, with one of the boundary components being of Dirichlet-type and the other two being of Neumann-type.
\end{enumerate}
\begin{lemma}\label{lemma:blocks}
	Each simple block type as above can be realized as a block (as defined in definition \ref{def:block}).
\end{lemma}

\subsection{Beginning of the proof of lemma \ref{lemma:blocks}} \label{section:blocks-proof-start}

\begin{figure}
	\centering
	\begin{tikzpicture}[x=0.75pt,y=0.75pt,yscale=-1,xscale=1]
		\node[inner sep=0pt] (imgafter) at (450,90) {\includegraphics[width=5cm]{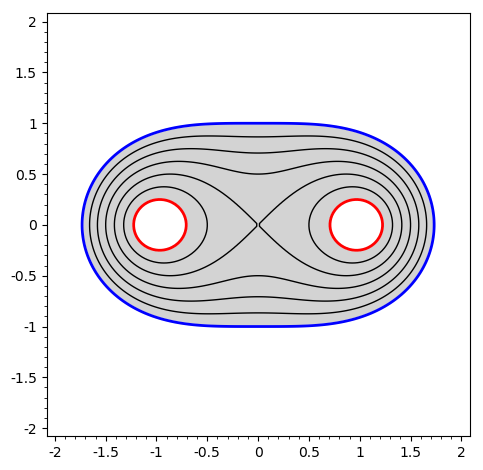}};
	\end{tikzpicture}
	\caption{The region $M$ (in gray) used to build the third simple block type, and level lines of the function $f_0$. The blue line is the level line $f_0=0$ (which will become the Dirichlet-type boundary component); the red lines are the level lines $f_0=1$ (which will become the Neumann-type boundary components).}
	\label{fig:third-simple-block-type}
\end{figure}

\begin{proof}[lemma \ref{lemma:blocks}]

	Choose a surface $M$ and an auxiliary function $f_0:M\to\left[0,1\right]$ according to the required simple block type as follows:
	\begin{itemize}
		\item For the first simple block type, let $M$ be the closed hemisphere $\left\{(x,y,z)\in\sph^2: z\geq 0\right\}$ and let $f_0:M\to\left[0,1\right]$ given by $f_0(x,y,z) = \frac z2$.
		\item For the second simple block type, let $M$ be the closed cylinder $\left[0,1\right]\times\sph^1$ and let ${f_0:\left[0,1\right]\times\sph^1\to\left[0,1\right]}$ be the projection map.
		\item For the third simple block type, let $M$ be the closed region (see figure \ref{fig:third-simple-block-type})
		\[
		M = \left\{z\in\CC:\left|z^2-1\right|\in\left[\frac12, 2\right]\right\}
		,
		\]
		and let $f_0:M\to\left[0,1\right]$ be
		\[
		f_0(z) = \frac{2 - \left|z^2-1\right|}{2 - \frac12}
		\,.
		\]
	\end{itemize}
	In each of the three cases, we set $f = 1-\left(1-f_0\right)^2$. Then the pair $(M, f)$ is diffeomorphic to the requested simple block.
	Note the following, regarding the critical points of $f_0$:
	\begin{itemize}
		\item In the first case, $f_0$ has a single critical point, which is a maximum point and belongs to the interior of $M$;
		\item In the second case, $f_0$ has no critical point;
		\item In the third case, $f_0$ has a single critical point, which is a saddle point and belongs to the interior of $M$.
	\end{itemize}
	Therefore $f_0$ has at most one critical point in $M$, and that critical point (if exists) is in the interior of $M$. The boundary components of $M$ that correspond to the Dirichlet-type boundary components are those with $f_0=0$, and the boundary components that correspond to the Neumann-type boundary components are those with $f_0=1$.
	
	On a neighborhood of each boundary component of $M$ we will define a metric. Let $C\subset\partial M$ be a boundary component and let $\epsilon>0$ be small enough.
	Let $Y = \left[0,\pi\right]\times\sph^1$ be a cylinder with a standard metric and let $f_Y:Y\to\left[0,1\right]$ given by $f_Y = \sin x$ where $x\in\left[0,\pi\right]$ is the first coordinate. Then $f_Y$ is an eigenfunction of the minus Laplacian on $Y$ (given by $-\Delta_Y h(x,y) = -\partial_x^2 h -\partial_y^2 h $ where $y\in\sph^1$ is considered to be an element of $\RR/2\pi\mathbb{Z}$) with its usual metric (given by $ds^2=dx^2+dy^2$), with eigenvalue $\lambda=1$. Let $U_Y\subset Y$ be a small neighborhood of one of the boundary components of $Y$:
	\begin{itemize}
		\item If $C$ is a Dirichlet-type boundary component, then let $U_Y=f_Y^{-1}(\left[0,\epsilon\right))$
		\item If $C$ is a Neumann-type boundary component, then let $U_Y = f_Y^{-1}(\left(1-\epsilon, 1\right])$
	\end{itemize}
	Then there is a diffeomorphism $\phi_C$ from $U_Y$ to a small neighborhood $U_C$ of $C$ which pullbacks $\left.f\right|_{U_C}$ to $f_Y$. On that neighborhood of $C$ define a metric $g_C$ to be the pushforward of the regular metric on $Y$ along $\phi_C$.
	By construction, the resulting metric satisfies the second condition and the third condition of definition \ref{def:block}, and $\left.f\right|_{U_C}$ is an eigenfunction of the minus Laplacian with eigenvalue $\lambda=1$.
	
	Note that when $f_0$ has a critical point $x_0$ on $M$, then the Hessian matrix of $f$ with respect to some local coordinates near $x_0$ is either negative-definite (in the case of the first simple block type) or indefinite (in the case of the third block type).
	Therefore in those cases, we may define a metric on a small neighborhood $U_{x_0}$ of $x_0$ by first choosing a metric $g_0$ such that $-\Delta_{g_0}f >0$ on $U_{x_0}$ and then scaling $g_0$ by a positive scalar function to get a metric $g$ for which $-\Delta_g f = f $.
	
	In the following, by the notation $\left\{x_0\right\}$ we mean the set $\left\{x_0\right\}$ in the cases where $x_0$ exists (i.e. the first and third block types), and $\emptyset$ when $x_0$ does not exist (i.e. in the second block type).
	
	Let $U\subset M$ be an open neighborhood of $\partial M\cup\left\{x_0\right\}$ where we have already chosen a metric $g_U$ by the above construction.
	Note that, given a number $c>0$, we can replace $U$ by a smaller neighborhood of $\partial M\cup\left\{x_0\right\}$ such that there is a smooth measure on $M$ which is compatible on $U$ with $g_U$ and satisfies
	\[
	\int_M fd\mu = c
	.
	\]
	Applying this with the number
	\[
	c =  - \int_{\partial M}\left(\nabla_{g_U} f\right)\ddn > 0
	,
	\]
	it follows that we can assume the existence of a smooth measure $\mu$ which is compatible on $U$ with $g_U$ and satisfies
		\begin{equation}\label{eq:block-full-integral}
			 \int_M fd\mu = -  \int_{\partial M}\left(\nabla_{g_U} f\right)\ddn
			 .
			\end{equation}
	
	It remains to extend the Riemannian metric from $U$ onto $M$ so that $f$ continues to be the eigenfunction of $-\Delta_g$. This is the most technical part of our argument. At this point we make a break in the proof of lemma \ref{lemma:blocks} to formulate theorem \ref{thm:interpolation-metric}, which provides us with a tool needed for such an extension. Then we complete the proof of lemma \ref{lemma:blocks}. Theorem \ref{thm:interpolation-metric} will be proved in section \ref{section:interpolation-preparations}.
\end{proof}
\subsection{Interpolation of the Riemannian metric}
\label{section:interpolation-formulation}
Let $M$ be a smooth orientable surface, $U\subset M$ an open subset equipped with a Riemannian metric $g_U$ and with $M\setminus U$ compact, and $f:M\to\RR$ be a smooth function without critical points. Say that a subset $R\subset M$ is an $(U,f)$-\emph{approximate down set} if it satisfies the following conditions:
\begin{itemize}
	\item $\overline R$ is compact,
	\item $\partial R$ is a finite disjoint union of piecewise smooth simple curves, and
	\item in the neighborhood of any point $p\in \partial R \setminus U$ the set $R$ coincides with $f^{-1}\left(\left(-\infty,f(p)\right)\right)$ ; in particular, $f$ is locally constant on $\partial R\setminus U$ (however not necessarily on $\partial R$).
\end{itemize}
Say that $\mu$ is $(g_U, f, \lambda)$-\emph{admissible} (or just \emph{admissible} if $g_U,f,\lambda$ are obvious in context) if, for every $(U,f)$-approximate down-set $R\subset M$ we have the inequality
\begin{equation}\label{eq:downset-inequality}
	-\int_R \lambda f d\mu \geq \int_{\partial R\cap U} \left(\nabla_{g_U} f\right) \ddn
\end{equation}
with equality iff $\partial R\setminus U$ has $1$-dimensional Hausdorff measure $0$.
We will prove the following theorem.
\begin{theorem} \label{thm:interpolation-metric}
	Let $M$ be a smooth orientable surface with a smooth measure $\mu$ (as defined in section \ref{section:notation}), and $f : M\to\RR$ a smooth function without critical points. Suppose that there is an open set $U\subset M$ and an Riemannian metric $g_U$ on $U$ compatible with $\mu$ such that $M\setminus U$ is compact and $\left.f\right|_U$ is a Laplacian eigenfunction of $g_U$ with eigenvalue $-\lambda$.
	
	Additionally, suppose that $\mu$ is $(g_U, f, \lambda)$-admissible.
	
	Then there is a metric on $M$, equal to $g_U$ on some $U'\subset U$ with $M\setminus U'$ compact, and compatible with $\mu$, for which $f$ is a Laplacian eigenfunction with eigenvalue $-\lambda$.
\end{theorem}
Note that the assumption that $\mu$ is $(g_U, f, \lambda)$-admissible is also necessary for the conclusion of theorem \ref{thm:interpolation-metric}. Indeed, if there is such a metric $g$ on $M$ and $R\subset M$ is a $(U,f)$-approximate down-set then we have
\begin{multline*}
-\int_R \lambda f d\mu
= \int_R \Delta_g f
= \int_{\partial R} \left(\nabla_g f\right)\ddn
=\\=
 \int_{\partial R\cap U} \left(\nabla_g f\right)\ddn
 +
  \int_{\partial R\setminus U} \left(\nabla_g f\right)\ddn
  \geq
  \int_{\partial R\cap U} \left(\nabla_g f\right)\ddn
  ,
\end{multline*}
where the last inequality follows from $R$ being approximate down-set, and it is an equality iff $\partial R\setminus U$ has $1$-dimensional Hausdorff measure $0$.

\subsection{End of the proof of lemma \ref{lemma:blocks}}

	Let $M, f, \mu, U, x_0$ be as in section \ref{section:blocks-proof-start}.
	
\begin{lemma}
	\label{lemma:interpolation-admissibility}
	For each $M,f, U, g_U$ as in section \ref{section:blocks-proof-start} with a smooth measure $\mu$ satisfying \eqref{eq:block-full-integral}, we can replace $U$ with a smaller neighborhood of $\partial M\cup\left\{x_0\right\}$ so that $\mu$ is $(g_U, f, 1)$-admissible.
\end{lemma}

\begin{proof}[End of proof of lemma \ref{lemma:blocks}]
Using lemma \ref{lemma:interpolation-admissibility} we can shrink $U$ further so that $\mu$ becomes $(g_U, f, 1)$-admissible. This allows us to apply theorem \ref{thm:interpolation-metric} to the surface $M_0 = M\setminus\left(\partial M\cup\left\{x_0\right\}\right)$ with the function $\left.f\right|_{M_0}$, the open set $U\cap M_0\subset M_0$, in order to get an extension $g$ of $g_U$ to the whole $M$. Then $M$ with the metric $g$ and the function $f$ is the needed building block.
\end{proof}

\begin{proof}[lemma \ref{lemma:interpolation-admissibility}]
	First we may assume that each connected component of $U$ intersects (and also contains) exactly one connected component of $\partial M\cup\left\{x_0\right\}$.
	
	For a piecewise smooth region $R\subset M$, denote
	\[
	S(R) = \int_R   fd\mu + \int_{\partial R\cap U} \left(\nabla_{g_U} f\right)\ddn
	\]
	We need to show how to make $U$ smaller so that any $(U,f)$-approximate down-set $R\subset M$ satisfies $S(R)\leq 0$, with equality if and only if $\partial R\setminus U$ has $1$-dimensional Hausdorff measure $0$.
	
	Note that $S$ is additive (in the sense that $R_1\cap R_2=\emptyset \implies S(R_1)+S(R_2)=S\left(R_1\cup R_2\right)$), and that if $\overline{R}\subset U$ then $S(R)=0$.
	
	In the case of the first simple block type, $\partial M\cup\left\{x_0\right\}$ consists of $\partial M$ which is a single Dirichlet-type boundary component, and a maximum point $x_0$ of $f$ with $f(x_0)=\frac 12$. Let $U_D, U_0$ be the connected components of $U$ that correspond to $\partial M$, $x_0$ respectively. Note that we may assume that $U_D = f^{-1}\left(\left[0,\epsilon_1\right)\right)$ and $U_0 = f^{-1}\left(\left(\frac12-\epsilon_2,\frac12\right]\right)$ for some small $0<\epsilon_1,\epsilon_2$. This implies that for any $(U,f)$-approximate down-set $R\subset M$, any connected component of $\partial M$, which intersects one of $U_D,U_0$, is contained in it. Therefore, to show that $\mu$ is $(g_U,f,1)$-admissible, it is enough to consider $(U,f)$-approximate down-set $R\subset M$ such that $R\cap U_D\in\left\{\emptyset,U_D\right\}$ (if there is a connected component $K$ of $R$ which is a subset of $U_D$ we may replace $R$ with $R\setminus K$; if there is a connected component $K$ of $M\setminus R$ which is a subset of $U_D$ we may replace $R$ with $R\cup K$) and similarly $R\cap U_0\in\left\{\emptyset,U_0\right\}$. Let us verify that $\mu$ is $(g_U,f,1)$-admissible in this case:
	\begin{itemize}
		\item If $U_0\subset R$ then we must have $R=M$ and we get (using \eqref{eq:block-full-integral}):
		\[
		S(R) = S(M) = \int_M   fd\mu + \int_{\partial M} \left(\nabla_{g_U} f\right)\ddn = 0
		\]
		\item If $U_0\cap R=\emptyset$ and $\partial R\not\subset U$ we must have $R=f^{-1}(\left[0,a\right))$ for some $0< a <\frac12$, so $\partial R\cap U=\partial M$ and we get
		\[
		S(R) = \int_R  fd\mu + \int_{\partial M} \left(\nabla_{g_U} f\right)\ddn
		=
		\int_R  fd\mu - \int_M   fd\mu < 0
		\]
		\item If $U_0\cap R=\emptyset$ and $\partial R\subset U$ we must have $R\subset U_D$ and we get $S(R)=0$.
	\end{itemize}
	
	The case of the second simple block type is similar: $\partial M\cup\left\{x_0\right\}$ consists of a Dirichlet-type boundary component $C_D\subset \partial M$ and a Neumann-type boundary component $C_N\subset\partial M$, so $U$ has two connected components $U_D\supset C_D$ and $U_N\supset C_N$. We may assume that $U_D = f^{-1}\left(\left[0,\epsilon_1\right)\right)$ and $U_N = f^{-1}\left(\left(1-\epsilon_2,1\right]\right)$ for some $0<\epsilon_1,\epsilon_2$. Therefore we get that any boundary component of a $(U,f)$-approximate down-set is either a subset of $U_D$ or disjoint to $U_D$, and similarly with $U_N$. Therefore it is enough to consider $(U,f)$-approximate down-sets $R\subset M$ such that $R\cap U_D\in\left\{\emptyset,U_D\right\}$ and $R\cap U_N\in\left\{\emptyset,U_N\right\}$. Let us verify that $\mu$ is $(g_U,f,1)$-admissible in this case:
	\begin{itemize}
		\item If $U_N\subset R$ then we must have $R=M$ and we get (using \eqref{eq:block-full-integral}):
		\[
		S(R) = S(M) = \int_M   fd\mu + \int_{\partial M} \left(\nabla_{g_U} f\right)\ddn = 0
		\]
		\item If $U_N\cap R=\emptyset$ and $\partial R\not\subset U$ we must have $R=f^{-1}(\left[0,a\right))$ for some $\epsilon_1\leq a \leq 1-\epsilon_2$ and we get
		\begin{multline*}
		S(R) = \int_R  fd\mu + \int_{C_D} \left(\nabla_{g_U} f\right)\ddn
		=\\=
		\int_R  fd\mu + \int_{\partial M} \left(\nabla_{g_U} f\right)\ddn
		=
		\int_R  fd\mu - \int_M   fd\mu < 0
		\end{multline*}
		\item If $U_N\cap R=\emptyset$ and $\partial R\subset U$ we must have $R\subset U_D$ and we get $S(R)=0$.
	\end{itemize}
	
	In the case of the third simple block type, $\partial M\cup \left\{x_0\right\}$ consists of one Dirichlet-type boundary component $C_D$, two Neumann-type boundary components $C_{N,1}, C_{N,2}$, and a saddle point $x_0$. Let $U_D\supset C_D$, $U_{N,i}\supset C_{N,i}$, $U_0\ni x_0$ be the corresponding connected components of $U$. As before we may assume that $U_D = f^{-1}\left(\left[0,\epsilon_1\right)\right)$ and $U_{N,0}\cup U_{N,1} = f^{-1}\left(\left(1-\epsilon_2,1\right]\right)$ for some $0<\epsilon_1,\epsilon_2$. We also may assume that $U_0$ is contractible. As above it is enough to consider $(U,f)$-approximate down-sets $R\subset M$ such that $R\cap U_D\in\left\{\emptyset, U_D\right\}$ and $R\cap U_{N,i}\in\left\{\emptyset, U_{N,i}\right\}$ for $i=1,2$, and also that there is no boundary component of $R$ which is a subset of $U_0$. Additionally, note that
	\begin{multline}\label{eq:path-estimation}
	\left|\int_{\partial R\cap U_0} \left(\nabla_{g_U} f\right)\ddn\right|
	=
	\left|
	-\int_{R\cap U_0}  fd\mu - \int_{\partial\left(R\cap U_0\right)\setminus\partial R}\left(\nabla_{g_U} f\right)\ddn
	\right|
	\leq \\ \leq
	 \mu(U_0) \sup_{U_0} \left| f\right|  + L_{g_U}\left(\partial U_0\right) \sup_{U_0} \left\|d f\right\|_{g_U}
	\end{multline}
	where $L_{g_U}\left(\partial U_0\right)$ is the perimeter of $U_0$ (measured using the metric $g_U$). By replacing $U_0$ by a smaller neighborhood of $x_0$ we can make the RHS of \eqref{eq:path-estimation} to be as close to $0$ as necessary.
	On the other hand, denoting $b(U_0) := \sup_{U_0} f<1-\epsilon_2$, the value
	\[
	\min\left\{\int_{W}   fd\mu : W\text{ is a connected component of } f^{-1}\left(\left(b(U_0),1\right]\right)\right\}
	\]
	can only increase when making $U_0$ smaller. Therefore we may assume that for any $R\subset M$ and for any connected component $W$ of $f^{-1}\left(\left(b(U_0),1\right]\right)$ we have
	\[
	\left|\int_{\partial R\cap U_0} \left(\nabla_{g_U} f\right)\ddn\right|
	<
	\int_W   fd\mu
	.
	\]
	\begin{itemize}
		\item If $\partial R\cap U_0\ne\emptyset$ then $R$ must be disjoint from some connected component $W$ of $f^{-1}\left(\left(b(U_0),1\right]\right)$, and also we must have $U_D\subset R$, therefore we get
		\begin{multline*}
		S(R) = \int_R   fd\mu + \int_{\partial R\cap U}\left(\nabla_{g_U} f\right)\ddn
		\leq \\ \leq
		\int_M   fd\mu - \int_W   fd\mu
		+ \int_{\partial R\cap U_0}\left(\nabla_{g_U} f\right)\ddn
		+ \int_{C_D}\left(\nabla_{g_U} f\right)\ddn
		=\\=
		\underbrace{
		\left(
		\int_M   fd\mu
		+ \int_{C_D}\left(\nabla_{g_U} f\right)\ddn
		\right)
	}_{=0\text{, as }\nabla_{g_U}f=0\text{ on }C_{N,i}}
		- \int_W   fd\mu
		+ \int_{\partial R\cap U_0}\left(\nabla_{g_U} f\right)\ddn
		<0
		.
		\end{multline*}
		\item If $U_{N,1}\subset R$ and $U_{N,2}\subset R$ then we must have $R=M$ and we get
		\[
		S(R) = S(M) = \int_M   fd\mu + \int_{\partial M}\left(\nabla_{g_U} f\right)\ddn=0
		.
		\]
		\item If for some $i=1,2$ we have $U_{N,i}\cap R=\partial R\cap U_0=\emptyset$ and $U_D\subset R$ then we have
		\[
		\int_{\partial R\cap U}\left(\nabla_{g_U} f\right)\ddn =
		\int_{C_D}\left(\nabla_{g_U} f\right)\ddn =
		 \int_{\partial M}\left(\nabla_{g_U} f\right)\ddn
		 .
		\]
		Therefore
		\begin{multline*}
		S(R) = \int_R   fd\mu + \int_{\partial R\cap U}\left(\nabla_{g_U} f\right)\ddn
		<\\< \int_M   fd\mu + \int_{\partial R\cap U}\left(\nabla_{g_U} f\right)\ddn
		= \int_M   fd\mu + \int_{C_D} \left(\nabla_{g_U} f\right)\ddn = 0
		.
		\end{multline*}
		\item If $U_D\cap R = \emptyset$ then we must have $R=\emptyset$ and we get $S(R)=0$.
	\end{itemize}
\end{proof}


\section{Construction using blocks}
We will show that lemma \ref{lemma:blocks} implies theorem \ref{thm:ovals}.

Note that given two blocks with underlying Riemann surfaces $(M_i,g_i)$ and boundary components $C_i\subset\partial M_i$, the surfaces $(M_i,g_i)$ can be smoothly glued together along $C_1,C_2$ such that the correspondence between $C_1$ and $C_2$ is isometry.
Indeed, let $\epsilon>0$ be small enough such that there are neighborhoods $C_i\subset U_i\subset M_i$ with isometries $\phi_{C_i}:U_i\to\left[0,\epsilon\right)\times \sph^1$. Let $\rho:\left(-\epsilon,\epsilon\right)\times\sph^1\to\left(-\epsilon,\epsilon\right)\times\sph^1$ be the isometry given by $\rho(x,y)=(-x,y)$. Then we can glue $M_1\cup M_2\cup\left(-\epsilon,\epsilon\right)$ along the isometries
\begin{align*}
	\phi_{C_1}&:U_1\to \left[0,\epsilon\right) \\
	\rho\circ\phi_{C_2}&:U_2\to\left(-\epsilon,0\right]
\end{align*}
and the result is obviously a smooth Riemannian surface $M_3$ with boundary which contains an isometric copy of each one of $M_1, M_2$ such that their union is the whole $M_3$ and their intersection is the image of $C_1,C_2$.

\begin{lemma}\label{lemma:gluing-blocks}
	Let $(M,g)$ be a closed Riemannian surface obtained from a disjoint union
	$\bigcup_{k=1}^n M_k$
	where each $M_k$ is a copy of a simple block from lemma \ref{lemma:blocks}, by gluing (as above) pairs of Dirichlet-type boundary components together and pairs of Neumann-type boundary components together.
	
	Let $f:M\to\RR$ be the function obtained by gluing the functions $f_i:M_i\to\RR$ that correspond to $M_i$.
	
	Assume that there is a function $\left\{1,\ldots,n\right\}\to\left\{1,-1\right\}$ such that if $M_i,M_j$ have boundary components glued together then $s(i)=s(j)$ if those are Neumann-type, or $s(i)=-s(j)$ if those are Dirichlet-type.
	
	Then $sf$ is an eigenfunction of the minus Laplacian on $M$, of eigenvalue $\lambda$ and we have ${\lambda \leq \lambda_{\alpha_1 n}}$ and $\lambda\operatorname{Vol}_g\left(M\right)\in\left[\alpha_2n,\alpha_3n\right]$.
	
	(Here by $sf:M\to\RR$ we mean $s(i)f(x)$ for $x\in M_i$)
\end{lemma}
\begin{proof}[lemma \ref{lemma:gluing-blocks}]
	Smoothness of $sf$ on the glued boundary components follows from property \ref{cond:block-boundary-func} of definition \ref{def:block}.
	That $sf$ is an eigenfunction of the minus Laplacian of eigenvalue $\lambda=1$ follows from property \ref{cond:block-eigenfunc} of definition \ref{def:block}.
	The bounds on $\operatorname{Vol}_g(M)$ are obvious when taking $\alpha_2, \alpha_3$ being the minimal and maximal (respectively) volume of a simple block type.
	
	To show the bound $\lambda \leq \lambda_{\alpha_1 n}$,
	let $\Omega=\bigcup_k \left(M_k^\circ\right) \subset M$ be the union of the blocks' interiors. By the Dirichlet-Neumann bracketing (see proposition 3.2.12 in \cite{SpectralGeometry}),
	we have (remember that $M$ is closed)
	\[
	\lambda^{\text{Neumann}}_k(\Omega) \leq \lambda_k(M)
	\]
	for any $k\in\mathbb{N}$, where $\lambda^{\text{Neumann}}_k(\Omega)$ is the $k$-th Neumann eigenvalue of $\Omega$.
	Rephrased using counting functions, we have for any $x\in\RR$,
	\[
	N^{\text{Neumann}}_\Omega (x) \geq N_M(x)
	\]
	where $N^{\text{Neumann}}_\Omega (x)$ (or $N_M(x)$) is the number of Neumann eigenvalues of $\Omega$ (or $M$) less than $x$.
	As $\Omega$ is the disjoint union of interiors of $n$ blocks, say the $i$-th block type occurs $n_i$ times with $n=n_1+n_2+n_3$, we have
	\[
	N^{\text{Neumann}}_\Omega(x) = 
	n_1 N^{\text{Neumann}}_{\text{1st block type}}(x)+
	n_2 N^{\text{Neumann}}_{\text{2nd block type}}(x)+
	n_3 N^{\text{Neumann}}_{\text{3rd block type}}(x)
	.
	\]
	In particular for $x=\lambda$ we get
	\[
	N^{\text{Neumann}}_\Omega(\lambda) \leq \left( \max_i N^{\text{Neumann}}_{i\text{th block type}}(\lambda)\right)\cdot n = \alpha_1 n
	,
	\]
	so $\lambda \leq \lambda_{\alpha_1 n}$.

\end{proof}

\begin{lemma}
	\label{lemma:what-to-glue}
	Let $n,k$ be nonnegative integers with $n\geq 1$, $k\leq 2$. Let $M$ be a surface with boundary, homeomorphic to a sphere with $n+k$ holes. Out of the $n+k$ boundary components of $M$, mark $n$ of them as Dirichlet-type and the other $k$ as Neumann-type. Then one can build a block $M'$ (as in definition \ref{def:block}) by gluing (as in lemma \ref{lemma:gluing-blocks}) several copies of the simple block types together along pairs of Neumann-type boundaries, such that $M'$ is homeomorphic to $M$ and this homeomorphism preserves the partition of boundary components of $M,M'$ into Dirichlet-type and Neumann-type.
\end{lemma}
\begin{proof}[lemma \ref{lemma:what-to-glue}]
	We will prove this by induction on $n$. If $n=1$ then a single simple block is enough: for $k=0$ it is the first type; for $k=1$ it is the second type; for $k=2$ it is the third type.
	
	Assume $n\geq 2$. Let $\gamma\subset M$ be a simple loop such that on any side of $\gamma$ there is at least $1$ Dirichlet-type boundary component and at most $1$ Neumann-type boundary component. Let $M_1,M_2$ be the connected components of $M\setminus \gamma$. On each of $M_1,M_2$, mark the boundary component that is the image of $\gamma$ as Neumann-type. With this marking, on each one of $M_1,M_2$ there are at most $2$ Neumann-type boundary components, and at most $n-1$ Dirichlet-type boundary components. Therefore, by induction hypothesis, $M_i$ ($i=1,2$) is homeomorphic to some gluing $M'_i$ of simple block types along Neumann-type boundary components. Gluing $M'_1$ and $M'_2$ along the image of $\gamma$ is then homeomorphic to $M$.
\end{proof}

\begin{proof}[theorem \ref{thm:ovals}]
	By lemma \ref{lemma:gluing-blocks} it is enough to decompose $\sph^2$ into simple blocks, glued together, such that the the image in $\sph^2$ of the Dirichlet-type boundry components of the simple blocks is equivalent to the given configuration $C$.
	
	By cutting $\sph^2$ along the ovals in $C$, it is enough to decompose manifolds of the form $M_m=\sph^2\setminus \left(D_1\cup\ldots \cup D_m\right)$ where $m\geq 1$ and $D_1,\ldots,D_m\subset \sph^2$ are disjoint open discs, such that in the decomposition of $M_m$, all the Dirichlet-type boundary components are in $\partial M_m$ (unglued) and all the Neumann-type boundary components are glued in the interior of $M_m$.
	But this is just lemma \ref{lemma:what-to-glue} applied on $M_m$ with all the boundary components $D_1,\ldots,D_m$ are declared Dirichlet-type.
	
	
	Note that the number of simple blocks needed for $\sph^2$ is $2n$ where $n$ is the number of ovals in $C$, because each simple block type has exactly one Dirichlet-type boundary and each oval in $C$ is a boundary of exactly two simple blocks. Therefore the bounds we get from lemma \ref{lemma:gluing-blocks}, while naively using the number of simple blocks, are true also when using the number of ovals, up to change of constants.
	
	The inequality $\lambda_n\leq\lambda$ is the Courant nodal domain theorem \cite{Courant}.
	
	Note that the Gaussian curvature $\kappa_g(p)$ of $g$ at a point $p\in\sph^2$ is obviously bounded by the supremal curvatures of the simple blocks
	\[
	\left|\kappa_g(p)\right|\leq \max_{S\text{ is a simple block}}\sup_{q\in S} \left|\kappa_S(q)\right| = \alpha_4\lambda
	\]
	where $\alpha_4$ is an absolute constant.
\end{proof}
\section{Interpolation of metric - preparations}
\label{section:interpolation-preparations}
\subsection{Idea} 
Let $M$ be a smooth orientable manifold with a smooth measure $\mu$, and let $f:M\to\RR$ be a smooth function on $M$ without any critical points. Then, according to lemma \ref{lemma:pointwise}, choosing a smooth Riemannian metric on $M$ is equivalent to choosing a smooth vector field which will become $\nabla_g f = g^{ij}\partial_j f$.

Note that $f$ being an eigenfunction of the minus Laplacian with an eigenvalue $\lambda$ is equivalent to $\nabla f$ having divergence $-\lambda f$.
Therefore to prove theorem \ref{thm:interpolation-metric} it is enough to find a smooth vector field $u$, which satisfies the following conditions:
\begin{enumerate}
	\item \label{cond:positivity} $\partial_u f>0$
	\item $\nabla \cdot u = -\lambda f$ 
	\item For some $U'\subset U$ with $M\setminus U'$ compact, $\left. u\right|_{U'}$ agrees with $\nabla_{g_U} f$.
\end{enumerate}

Given such $u$, we can apply lemma \ref{lemma:pointwise} with $\omega = df$ to recover the Riemannian metric $g$ for which $u = \nabla_g f$ and it would satisfy the claimed properties in the conclusion of theorem \ref{thm:interpolation-metric}.

The process of building $u$ is divided to three steps:
\begin{itemize}
	\item Setting the values of integrals of $\partial_u f$ over level lines of $f$.
	\item Building a smooth function $h:=\partial_u f > 0$ such that its integrals over level lines are as given by the first step.
	\item Building a vector field $u$ such that $\partial_u f$ is as given by the second step and $\nabla\cdot u = -\lambda f$.
\end{itemize}
To do this it is more convenient to look at the quotient of $M$ given by identifying points which belong to the same level set of $f$. The resulting topological space is a certain kind of $1$-dimensional non-Hausdorff manifold, which we will term "blueprint".

\subsection{Blueprints}
Here we will use the following notion of non-Hausdorff manifolds:
\begin{definition}
	A \emph{non-Hausdorff manifold} of dimension $n$ is a topological space $X$ which is locally homeomorphic to $\RR^n$ at every point.
\end{definition}
One well-known example of a $1$-dimensional non-Hausdorff manifold are the \emph{line with two origins}, obtained from two copies of $\RR$ by identifying all corresponding points of the copies but the origin. Another well-known example is the \emph{branching line} obtained by identifying the corresponding points which are $<0$ in the two copies of $\RR$.
\begin{definition}
	Given a $1$-dimensional non-Hausdorff manifold $X$, we will say that a point $p\in X$ is \emph{singular} if there is another point $q\in X$ such that $q\neq p$ but every neighborhood of $p$ intersects every neighborhood of $q$. Otherwise, we will say that $p$ is \emph{regular}. We will denote by $\sing{X}$ the set of singular points of $X$.
\end{definition}
\begin{definition}
	A \emph{blueprint} is a pair of a $1$-dimensional non-Hausdorff manifold $X$, and a continuous map $\pi:X\to\RR$ such that:
	\begin{enumerate}
		\item $X$ has finitely many singular points
		\item For any point $p\in X$ there is a neighborhood $p\in I\subset X$ (which can be choosen to be homeomorphic to an interval) such that $\left. \pi\right|_I:I\to \RR$ is injective
		\item The image $\pi(X)\subset \RR$ of $\pi$ is bounded.
	\end{enumerate}
\end{definition}

The following lemma shows that blueprints are obtained as a quotient of a surface by level sets. Here a point $p$ is called a \emph{weak local minimum} of a function $f$ if for some neighborhood $U$ of $p$ we have $\forall q\in U: f(q)\geq f(p)$. A \emph{weak local maximum} is defined analogously.

\begin{lemma}\label{lemma:quotient-is-blueprint}
	Let $M$ be a connected smooth compact surface possibly with boundary, $f:M\to\RR$ a smooth function, and $U\subset M$ an open subset of $M$ such that $\overline{U}$ does not contain any critical point of $f$ and $\overline{U}\cap\partial M=\emptyset$. Assume that $\left. f\right|_{\partial U}:\partial U\to\RR$ has only finitely many points which are (weak) local minimum or maximum.
	Let $\sim_{f}$ be the equivalence relation on $U$ which is defined as $x\sim_f y$ for $x,y\in U$ iff there is a path in $U$ between $x$ and $y$ on which $f$ is constant.
	Then $U/\sim_f$ together with the function 
	$\pi_f:U/\sim_f \to\RR$
	which corresponds to $\left. f\right|_U:U\to\RR$
	is a blueprint.
\end{lemma}
\begin{proof}
	First we note that the quotient map $Q: U\to U/\sim_f$ is an open function. For this it is enough to show that for any $p,q\in U$ such that $p\sim_f q$ and for any neighborhood $U_p\subset U$ of $p$, there is a neighborhood $U_q\subset U$ of $q$ such that any $q'\in U_q$ is $\sim_f$-equivalent to some $p'\in U_p$. This is obvious considering a neighborhood of the path that shows $p\sim_f q$.
	
	Now we will show that $U/\sim_f$ is locally homeomorphic to $\RR$ and that $\pi_f$ satisfies condition 2.
	As we assume that $f$ has no critical points in $U$, it follows that $\left. f\right|_U:U\to\RR$ is an open map. Additionally, we get that for any point $p\in U$ there is a neighborhood $U_p\subset U$ such that:
	\[
	\forall x,y\in U_p: x\sim_f y \iff f(x)=f(y)
	\]
	Using this choice of $U_p$, we have that $f$ induces a homeomorphism between $U_p/\sim_f$ and an open interval $I=f(U_p)\subset \RR$ (here we denote by $U_p/\sim_f$ the quotient space formed from $U_p$ and the equivalence relation $\sim_f$ restricted on $U_p$, in order to distinguish it from the subspace ${Q(U_p)\subset U/\sim_f}$ which a priori could have a different topology). This homeomorphism factors as
	\[
	U_p/\sim_f \to Q(U_p) \to I
	\]
	with intermediate maps being bijective. Therefore we get that $Q(U_p)$ is homeomorphic by $\pi_f$ to $I$.
	Additionally, as $Q$ is an open map, we get that $Q(U_p)$ is open in $U/\sim_f$. Therefore $Q(U_p)$ is a neighborhood of $Q(p)$ which is homeomorphic to $I$ (and therefore to $\RR$). This means that $U/\sim_f$ is locally homeomorphic to $\RR$. Additionally $\pi_f$ is injective on $Q(U_p)$, which means that $f$ satisfies condition 2.
	
	Note that the image $f(U)\subset\RR$ is clearly bounded because $M$ is compact.
	Therefore it is enough to show that $U/\sim_f$ has finitely many singular points. For this note that if $p,q\in U$ correspond to points in $U/\sim_f$ which are distinct but not separated by neighborhood, then there is a path $\gamma$ in $\overline{U}$ from $p$ to $q$ such that $f\circ\gamma$ is constant and any point in $\gamma\cap\partial U$ is a weak local extremal point of $\left.f\right|_{\partial U}$. By our assumption, there are only finitely many such extremal points, and from this it follows that $U/\sim_f$ has finitely many singular points.
\end{proof}

\begin{definition}
	Given a blueprint $\pi:X\to\RR$ and two points $p_0,p_1\in{X}$,
	\begin{enumerate}
		\item we will define a \emph{left half-neighborhood} of $p_0$ to be the intersection of a neighborhood of $p_0$ in $X$ and the $\pi$-preimage of a left half-neighborhood of $\pi(p_0)\in\RR$. Similarly define \emph{right half-neighborhood}.
		\item a \emph{punctured left half-neighborhood} of $p_0$ is a set $C$ such that $C\cup\left\{p_0\right\}$ is a left half-neighborhood of $p_0$. Similarly define \emph{punctured right half-neighborhood}.
		\item we will say that $p_0\sim_{+}p_1$ iff every right half-neighborhood of $p_0$ intersects every right half-neighborhood of $p_1$. Similarly, we will say that $p_0\sim_{-} p_1$ iff every left half-neighborhood of $p_0$ intersects every left half-neighborhood of $p_1$.
	\end{enumerate}
\end{definition}

\begin{figure}
	\centering
\begin{tikzpicture}[x=0.75pt,y=0.75pt,yscale=-1,xscale=1]
	
	\draw    (100,82.7) -- (254,82.7) ;
	\draw  [fill={rgb, 255:red, 0; green, 0; blue, 0 }  ,fill opacity=1 ][line width=1.5]  (249.5,87.2) .. controls (249.5,84.71) and (251.51,82.7) .. (254,82.7) .. controls (256.49,82.7) and (258.5,84.71) .. (258.5,87.2) .. controls (258.5,89.69) and (256.49,91.7) .. (254,91.7) .. controls (251.51,91.7) and (249.5,89.69) .. (249.5,87.2) -- cycle ;
	\draw  [fill={rgb, 255:red, 0; green, 0; blue, 0 }  ,fill opacity=1 ][line width=1.5]  (249.5,78.2) .. controls (249.5,75.71) and (251.51,73.7) .. (254,73.7) .. controls (256.49,73.7) and (258.5,75.71) .. (258.5,78.2) .. controls (258.5,80.69) and (256.49,82.7) .. (254,82.7) .. controls (251.51,82.7) and (249.5,80.69) .. (249.5,78.2) -- cycle ;
	\draw    (258.5,78.2) .. controls (315,79.05) and (337,67.55) .. (377,37.55) ;
	\draw    (258.5,87.2) .. controls (328,87.55) and (341,83.4) .. (377.5,113.55) ;
	\draw  [fill={rgb, 255:red, 0; green, 0; blue, 0 }  ,fill opacity=1 ][line width=1.5]  (95.5,87.2) .. controls (95.5,84.71) and (97.51,82.7) .. (100,82.7) .. controls (102.49,82.7) and (104.5,84.71) .. (104.5,87.2) .. controls (104.5,89.69) and (102.49,91.7) .. (100,91.7) .. controls (97.51,91.7) and (95.5,89.69) .. (95.5,87.2) -- cycle ;
	\draw  [fill={rgb, 255:red, 0; green, 0; blue, 0 }  ,fill opacity=1 ][line width=1.5]  (95.5,78.2) .. controls (95.5,75.71) and (97.51,73.7) .. (100,73.7) .. controls (102.49,73.7) and (104.5,75.71) .. (104.5,78.2) .. controls (104.5,80.69) and (102.49,82.7) .. (100,82.7) .. controls (97.51,82.7) and (95.5,80.69) .. (95.5,78.2) -- cycle ;
	\draw    (3,34.55) .. controls (37,78.9) and (30.5,76.55) .. (95.5,78.2) ;
	\draw    (3,124.05) .. controls (43,94.05) and (42,86.05) .. (95.5,87.2) ;
	\draw    (-13.5,210.2) -- (416.5,210.2) ;
	\draw  [dash pattern={on 0.84pt off 2.51pt}]  (178.5,117.2) -- (178.5,191.05) ;
	\draw [shift={(178.5,193.05)}, rotate = 270] [color={rgb, 255:red, 0; green, 0; blue, 0 }  ][line width=0.75]    (10.93,-3.29) .. controls (6.95,-1.4) and (3.31,-0.3) .. (0,0) .. controls (3.31,0.3) and (6.95,1.4) .. (10.93,3.29)   ;
	
	\draw (92.5,100.2) node [anchor=north west][inner sep=0.75pt]   [align=left] {$p_2$};
	\draw (94.5,46.7) node [anchor=north west][inner sep=0.75pt]   [align=left] {$p_1$};
	\draw (248.5,47.2) node [anchor=north west][inner sep=0.75pt]   [align=left] {$p_3$};
	\draw (253,97.7) node [anchor=north west][inner sep=0.75pt]   [align=left] {$p_4$};
	\draw (184,143.7) node [anchor=north west][inner sep=0.75pt]   [align=left] {$\pi$};
	\draw (379.5,184.7) node [anchor=north west][inner sep=0.75pt]   [align=left] {$\RR$};
	
	\draw (379.5,80) node [anchor=north west][inner sep=0.75pt]   [align=left] {$X$};

\end{tikzpicture}
\caption{Example of a blueprint $\pi:X\to\RR$ with $\sing{X}=\left\{p_1,p_2,p_3,p_4\right\}$. This $X$ is built from two copies of $\RR$ by identifying the two copies of $\left(0,1\right)\subset \RR$ together.}
\label{fig:blueprint-example}
\end{figure}

For example, in figure \ref{fig:blueprint-example}, we have $p_1 \sim_{+} p_2$ and $p_3\sim_{-}p_4$, but $p_1\not\sim_{-}p_2$ and $p_3\not\sim_{+}p_4$.

Note that $\sim_{+}$ and $\sim_{-}$ are equivalence relations on ${X}$ (transitivity follows from the assumption that $\sing{X}$ is finite), and that if $p_0\sim_{+}p_1$ (or $p_0\sim_{-}p_1$) then either $p=q$ or $p,q\in\sing{X}$. Additionally, if $p\in\sing{X}$ and $\left(a_i\right)_i\subset X\setminus\sing{X}$ is a sequence that converges to $p$ with $\pi(a_i)$ monotonously decreasing (or increasing), then for any $q\in\sing{X}$, $a_i$ also converges to $q$ iff $p\sim_{+}q$ (or $p\sim_{-}q$ respectively).

Note that given a blueprint $\pi:X\to\RR$, we may define derivatives on functions over $X$ by using $\pi$ as a coordinate chart. This gives a canonical structure of smooth manifold to a blueprint.
\begin{definition}
	Given a blueprint $\pi:X\to\RR$, an \emph{embedded open interval} is a subset $I\subset X$ such that the restriction $\left.\pi\right|_I:I\to\RR$ is a homeomorphism to its image and its image $\pi(I)\subset \RR$ is an open interval in $\RR$.
\end{definition}
Note that for any continuous injection $i:(a,b)\to X$ from an open interval $(a,b)\subset\RR$, its image $i\left((a,b)\right)\subset X$ is an embedded open interval in $X$. Indeed, if $x\in(a,b)$ is a point such that $\pi\circ i:(a,b)\to\RR$ is not locally strictly monotonic near $x$ then, by the properties of a blueprint, $i:(a,b)\to X$ would not be locally injective near $x$.
\begin{definition}
	Given a blueprint $\pi:X\to\RR$, we will say that a function $\phi:X\to\RR$ is \emph{pseudosmooth} if for any closed interval $I\subset\RR$ and any path-connected component $ U$ of $\pi^{-1}(I)$, the function $\widetilde{\phi}_U:I\to\RR$ defined by
	\[
	\widetilde{\phi}_U(t) = \sum_{x\in U\cap \pi^{-1}(t)} \phi(x)
	\]
	is smooth on $I$.
	Note that the finiteness of the sum follows from the connectedness of $U$ and from the finiteness of $\sing{X}$.
\end{definition}
Note that, in the above definition, $U$ being a path-connected component of $\pi^{-1}(I)$ is equivalent to $U$ being a connected (relatively-)clopen subset of $\pi^{-1}(I)$, therefore also equivalent to $U$ being a connected component of $\pi^{-1}(I)$.

Additionally, note that pseudosmooth functions are in particular smooth on the regular points.
Also note that if $I\subset X$ is an embedded open interval and $\psi:\RR\to\RR$ is a smooth bump function with support in $\overline{\pi(I)}$ then the function $\phi:X\to\RR$ defined by
\[
\phi(x) = \begin{cases}
	\psi(\pi(x)) & x\in I \\
	0 & x\notin I
\end{cases}
\]
is pseudosmooth.

We will say that a function $g:I\to\RR$ on a finite open interval $I$ is \emph{strictly-smooth} if $g$ is smooth, and for every $k\geq 0$, the limit of $g^{(k)}$ at each endpoint of $I$ exists. Equivalently, $g:I\to\RR$ is strictly-smooth iff it can be continued to a smooth function $\overline{I}\to\RR$ (or, equivalently, to a smooth function $\RR\to\RR$).

We will say that a function $g:I\to\RR$ on a finite open interval $I$ is \emph{strictly-bump} if $g$ is smooth, positive on $I$, and for every $k\geq 0$, the limit of $g^{(k)}$ at the endpoints of $I$ is $0$.
Equivalently, $g:I\to\RR$ is strictly-bump iff its extension to $\RR$ by $0$ outside $I$ is smooth.

Given a strictly-bump function $g:I\to\RR$ where $I$ is an embedded open interval of a blueprint $X$, it will be convenient to extend $g$ by $0$ to $X\setminus I$. Note that this extension might not be smooth near some singular points of $X$, but it would always be pseudosmooth.

The following lemma describes an analogue of a partition of unity argument for blueprints:

\begin{lemma} \label{lemma:function-decomposition}
	Suppose that we have a blueprint $\pi:X\to\RR$, a positive pseudosmooth function $\phi:X\to\RR$ and a finite family $\mathcal{F}=\left\{\left(I_i, \chi_i\right)\right\}_{i}$ of embedded open intervals $I_i\subset X$ and strictly-bump functions $\chi_i:I_i\to\left(0,\infty\right)$ extended by $0$ to $X\setminus I_i$ such that:
	\begin{enumerate}
		\item \label{condition:function-decomposition:cover} $\left\{I_i\right\}_i$ is a cover of $X$;
		\item \label{condition:function-decomposition:boundary} there are strictly-smooth positive functions $h_i:I_i\to\RR$ such that for $x\in X$ outside some compact subset of $X$,
		\begin{equation}\label{eq:smooth-linear-combination}
			\sum_i h_i(x) \chi_i(x) = \phi(x)
			.
		\end{equation}
	\end{enumerate}
	Then there are strictly-smooth bounded positive functions $\widetilde h_i:I_i\to\RR$ such that \eqref{eq:smooth-linear-combination} holds (with $\widetilde h_i$ instead of $h_i$) for every $x\in X$.
\end{lemma}
\begin{proof}
	As a first step, we will show that we may assume that the set of intervals in $\mathcal{F}$ has a certain form. Let $I\subset X$ be an embedded open interval, and let $J_1,J_2\subset I$ be two open subintervals such that $I=J_1\cup J_2$ and $J_1,J_2$ have no common endpoints. That is, if we denote $\pi(I)=(a,b)\subset \RR$, $\pi(J_i)=(c_i,d_i)\subset \RR$, one of the following options are satisfied for some $i, j$ with $\left\{i,j\right\} =\left\{1,2\right\}$:
	\begin{enumerate}
		\item $a = c_i < c_j < d_j < d_i = b$
		\item $a = c_i < c_j < d_i < d_j = b$
	\end{enumerate}
	Let $\chi_{J_i}:J_i\to\left(0,\infty\right)$ be strictly-bump functions and let $\chi_I=\chi_{J_1}+\chi_{J_2}$ which is also strictly-bump. Let $\mathcal{F}$ be a family such that $\left(I,\chi_I\right)\in\mathcal{F}$ and let
	\[
	\mathcal{F}'=\left(\mathcal{F}\setminus\left\{\left(I,\chi_I\right)\right\}\right)\cup\left\{\left(J_i, \chi_{J_i}\right):i=1,2\right\}
	\]
	We claim that the statements of lemma \ref{lemma:function-decomposition} for $\mathcal{F}$ and for $\mathcal{F}'$ are equivalent. The cover assumptions are obviously equivalent. For the rest it is enough to show the following equality of set of functions $X\to\RR$:
	\begin{multline*}
	\left\{x\mapsto \sum_{i=1,2}h_{J_i}(x)\chi_{J_i}(x) : h_{J_i}:J_i\to\RR\text{ are strictly-smooth positive}\right\}
	=\\=
	\left\{x\mapsto h_{I}(x)\chi_{I}(x) : h_{I}:I\to\RR\text{ is strictly-smooth positive}\right\}
	\end{multline*}
	The $\supset$ direction follows immediately from $\chi_{J_1}+\chi_{J_2}=\chi_I$. The $\subset$ direction follows by defining $h_I = \frac{h_{J_1}\chi_{J_1}+h_{J_2}\chi_{J_2}}{\chi_I}$ and noting that the restriction on the endpoints of $J_1,J_2$ implies that $h_I$ is strictly-smooth.
	
	From the above it follows that, given a family $\mathcal{F}$, we may take any interval $I$ in it and break it up into two subintervals $J_1,J_2$ whose union is $I$ and have no common endpoints; and vice-versa, if $\mathcal{F}$ contains two intervals $J_1,J_2$ such that their union is also an embedded interval $I$ such that $J_1,J_2$ have no common endpoints in $I$ then we may replace $J_1,J_2$ with $I$.
	Using these operations we may assume the following about $\mathcal{F}$:
	\begin{enumerate}
		\item Every interval from $\mathcal{F}$ passes through at most $1$ singular point of $X$; otherwise, any interval which passes through more than $1$ singular points can be broken down to subintervals such that each one passes through at most $1$ singular point.
		\item For each singular point $p\in\sing{X}$ there is exactly $1$ interval from $\mathcal{F}$ which passes through $p$; otherwise, given two embedded intervals $I_1,I_2\ni p$, replace $\left\{I_1\right\}$ with $\left\{J_1,K,J_2\right\}$ such that $I_1=J_1\cup K\cup J_2$, $p\in K\subset I_1\cap I_2$ and $p\notin J_1,J_2$; then replace $\left\{K,I_2\right\}$ with $\left\{I_2\right\}$.
		\item The subset of intervals from $\mathcal{F}$ which do not pass through any singular point covers $X\setminus\sing{X}$; otherwise, for any $p\in\sing{X}$ and $I\ni p$ with $\pi(I)=(\pi(p)-u,\pi(p)+v)$, replace $\left\{I\right\}$ by 
		\[
		\begin{Bmatrix}
			I\cap\pi^{-1}\left((\pi(p)-u,\pi(p))\right), \\
			 I\cap\pi^{-1}\left((\pi(p)-\epsilon,\pi(p)+\epsilon)\right), \\
			  I\cap\pi^{-1}\left((\pi(p),\pi(p)+v)\right)
			\end{Bmatrix}.
			\]
		\item For any two singular points $p_0,p_1\in\sing{X}$ and the corresponding embedded intervals $I_{p_i}\ni p_i$ in $X$, the intervals $\pi\left(I_{p_0}\right),\pi\left(I_{p_1}\right)\subset\RR$ are either disjoint or equal; otherwise apply the previous transformation with $\epsilon$ small enough so that the same $\epsilon$ will be used for every singular point.
	\end{enumerate}
	
	
	
	That is,
	we may assume that $\mathcal{F}$ can be divided as
	\[
	\mathcal{F} = \mathcal{F}_{\mathrm{reg}}\cup \left\{(I_p,\chi_p):p\in \sing{X}\right\}
	\]
	where $I_p$ is a small enough neighborhood of $p$ such that for any two points $p_0,p_1\in\sing{X}$ the intervals $\pi(I_{p_0}), \pi(I_{p_1})$ are either disjoint or equal, and the intervals from $\mathcal{F}_\mathrm{reg}$ cover exactly the regular points.
	
	Let $K$ be an equivalence class of $\sim_{+}$ and let $U\subset X\setminus\sing{X}$ such that $U\cup\{p\}$ is a small enough right half-neighborhood of some $p\in K$ (note that the same $U$ will work for any choice of $p\in K$). Then to say that \eqref{eq:smooth-linear-combination} holds in $U$ is to say that
	\[
	\sum_{p\in K} \widetilde h_p(x)\chi_p(x) + \sum_{\left(I_i,\chi_i\right)\in\mathcal{F}_\mathrm{reg}} \widetilde h_i(x)\chi_i(x) = \phi(x)
	\]
	Note that the second sum has vanishing Taylor series at $K$, hence it fixes the sum, over $p\in K$, of the Taylor series of $\widetilde h_p(x)\chi_p(x)$ around $p$.
	
	Similarly, by reducing to a left half-neighborhood of points in an equivalence class of $\sim_{-}$, we get that the sum over an equivalence class of $\sim_{-}$ of Taylor series of $\widetilde h_p\chi_p$ is fixed. For \eqref{eq:smooth-linear-combination} to hold in some $x=p\in \sing{X}$ it just mean that $\widetilde h_p(p)\chi_p(p) = \phi(p)$. From the pseudosmoothness of $\phi$ it follows that all those linear constraints on the Taylor series around points in $\sing{X}$ are compatible, hence one can find, for any $p\in\sing{X}$, a Taylor series $T_p$ such that the constraints above are satisfied.

	For any $p\in\sing{X}$, let $\widetilde h_p:I_p\to\RR$ be a strictly-smooth bounded positive function such that the Taylor series of $\widetilde h_p(x)\chi_p(x)$ at $x=p$ is $T_p$.

	Let
	\[
	\phi_1(x) = \sum_{\left(I_i,\chi_i\right)\in\mathcal{F}_\mathrm{reg}}h_i(x)\chi_i(x) + \sum_{p\in \sing{X}}\widetilde h_p(x)\chi_p(x)
	.
	\]
	Then $\frac{\phi}{\phi_1}$ is continuous and flat at singular points, with value $1$ there. Additionally $\phi_1$ is positive, and agrees with $\phi$ outside a compact subset of $X$.
	Therefore the following decomposition of $\phi$ is as needed:
	\[
	\phi(x) = \sum_{\left(I_i,\chi_i\right)\in\mathcal{F}_\mathrm{reg}}\frac{h_i(x)\phi(x)}{\phi_1(x)}\chi_i(x) + \sum_{p\in \sing{X}}\frac{\widetilde h_p(x)\phi(x)}{\phi_1(x)}\chi_p(x)
	,
	\]
	proving lemma \ref{lemma:function-decomposition}.
	
\end{proof}

Let us say that a subset $C$ of a blueprint $X$ is simple if $C$ has finitely many connected components.
Note that, for any simple subset $C$, the set $\partial C$ is finite.
Indeed, if we assume that $X$ and $C$ are connected, we get that $X\setminus\sing{X}$ is homeomorphic to a disjoint union of finitely many copies of $\RR$, and that the intersection of each one of them with $C$ is a disjoint union of at most $2$ intervals. This implies that
\[
\left|\partial C\setminus\sing{X}\right|\leq 2\left|\left\{\text{connected components of }X\setminus\sing{X}\right\}\right|<\infty
\]

In particular, for any $x\in X$, 
there is some punctured left half-neighborhood $P$ of $x$ such that $P\cap\partial C=\emptyset$, which implies that either $P\subset C$ or $P\cap C=\emptyset$ (and similarly for punctured right half neighborhood).
Therefore one can define a finitely-supported function $d_C:X\to\left\{-1,0,1\right\}$ in the following way:
\begin{align}
	d^{\text{left}}_C(x) &= \begin{cases}
		1 & C\text{ contains some punctured left half-neighborhood of }x \\
		0 & C\text{ is disjoint to some punctured left half-neighborhood of }x
	\end{cases} \\
	d^{\text{right}}_C(x) &= \begin{cases}
		1 & C\text{ contains some punctured right half-neighborhood of }x \\
		0 & C\text{ is disjoint to some punctured right half-neighborhood of }x
	\end{cases} \\
	\label{eq:simple-set-boundary-orientation}
	d_C(x) &= d^{\text{left}}_C(x) - d^{\text{right}}_C(x)
\end{align}

See figure \ref{fig:dvalues} for a demonstration of the value of $d_C(p)$ depending on how $C$ looks in a neighborhood of $p$. Note that $d_C(p)$ does not depend on whether $p\in C$ or $p\notin C$.

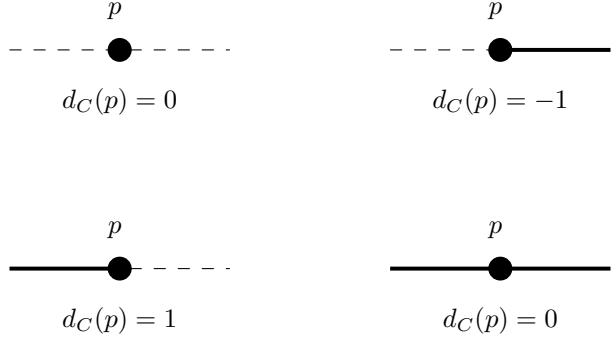
\begin{figure}
	\centering
\begin{tikzpicture}[x=0.75pt,y=0.75pt,yscale=-1,xscale=1]
	\def\horspace{190}
	\def\verspace{110}
	

	\def\pointx{130} 
	\def\pointy{42} 
	\def\pointr{5}
	
	\def\linelen{50}
	
	\def\textdx{130}
	\def\textdy{60}
	
	\def\textpx{123}
	\def\textpy{17}
	
	\draw [fill={rgb, 255:red, 0; green, 0; blue, 0 }  ,fill opacity=1 ][line width=1.5] (\pointx+\pointr,\pointy) arc (0:360:\pointr);
	\draw  [dash pattern={on 4.5pt off 4.5pt}]  (\pointx+\pointr,\pointy) -- (\pointx+\pointr+\linelen,\pointy) ;
	\draw  [dash pattern={on 4.5pt off 4.5pt}]  (\pointx-\pointr-\linelen,\pointy) -- (\pointx-\pointr,\pointy) ;
	
	\draw [fill={rgb, 255:red, 0; green, 0; blue, 0 }  ,fill opacity=1 ][line width=1.5] (\pointx+\pointr+\horspace,\pointy) arc (0:360:\pointr);
	\draw  [line width=1.5] (\pointx+\pointr+\horspace,\pointy) -- (\pointx+\pointr+\linelen+\horspace,\pointy) ;
	\draw  [dash pattern={on 4.5pt off 4.5pt}]  (\pointx-\pointr-\linelen+\horspace,\pointy) -- (\pointx-\pointr+\horspace,\pointy) ;
	
	\draw [fill={rgb, 255:red, 0; green, 0; blue, 0 }  ,fill opacity=1 ][line width=1.5] (\pointx+\pointr,\pointy+\verspace) arc (0:360:\pointr);
	\draw  [dash pattern={on 4.5pt off 4.5pt}]  (\pointx+\pointr,\pointy+\verspace) -- (\pointx+\pointr+\linelen,\pointy+\verspace) ;
	\draw   [line width=1.5] (\pointx-\pointr-\linelen,\pointy+\verspace) -- (\pointx-\pointr,\pointy+\verspace) ;
	
	\draw [fill={rgb, 255:red, 0; green, 0; blue, 0 }  ,fill opacity=1 ][line width=1.5] (\pointx+\pointr+\horspace,\pointy+\verspace) arc (0:360:\pointr);
	\draw  [line width=1.5] (\pointx+\pointr+\horspace,\pointy+\verspace) -- (\pointx+\pointr+\linelen+\horspace,\pointy+\verspace) ;
	\draw  [line width=1.5]  (\pointx-\pointr-\linelen+\horspace,\pointy+\verspace) -- (\pointx-\pointr+\horspace,\pointy+\verspace) ;

	\draw (\textpx,\textpy) node [anchor=north west][inner sep=0.75pt]   [align=left] {$p$};
	\draw (\textdx,\textdy) node [anchor=north][inner sep=0.75pt]   [align=center] {$ d_C(p)=0$};
	\draw (\textpx+\horspace,\textpy) node [anchor=north west][inner sep=0.75pt]   [align=left] {$p$};
	\draw (\textdx+\horspace,\textdy) node [anchor=north][inner sep=0.75pt]   [align=center] {$d_C(p)=-1$};
	\draw (\textpx,\textpy+\verspace) node [anchor=north west][inner sep=0.75pt]   [align=left] {$p$};
	\draw (\textdx,\textdy+\verspace) node [anchor=north][inner sep=0.75pt]   [align=center] {$d_C(p)=1$};
	\draw (\textpx+\horspace,\textpy+\verspace) node [anchor=north west][inner sep=0.75pt]   [align=left] {$p$};
	\draw (\textdx+\horspace,\textdy+\verspace) node [anchor=north][inner sep=0.75pt]   [align=center] {$d_C(p)=0$};

\end{tikzpicture}
\caption{Values of $d_C(p)$ given how $C$ looks like in a neighborhood of $p$. Here a full line represents points that belong to $C$; a dashed line represents points that do not belong to $C$. 
}
\label{fig:dvalues}
\end{figure}

In particular if $x\notin\partial C$ or $x\notin\partial\left(C\Delta\left\{x\right\}\right)$ (here $\Delta$ denotes the symmetric set difference operation) then $d_C(x)=0$.
The value of $d_C(x)$ can be thought of as the orientation of $\partial C$ at $x$, similarly to the boundary map from $1$-dimensional chains to $0$-dimensional chains in singular homology.

\begin{lemma}\label{lemma:resolving-homology}
	Let $X$ be a blueprint, and let $F$ be a map from the set of simple subset of $X$ to $\RR$, which satisfies the following conditions:
	\begin{enumerate}
		\item \label{condition:resolving-homology:clopen-to-zero} If $C$ is clopen then $F(C)=0$.
		\item \label{condition:resolving-homology:finite-to-zero} If $C$ is finite then $F(C) = 0$.
		\item \label{condition:resolving-homology:additivity} If $C_1,C_2$ are two disjoint simple subsets then $F(C_1\cup C_2) = F(C_1) + F(C_2)$.
		\item \label{condition:resolving-homology:positivity} (positivity) If $d_C(X) = \left\{0,1\right\}$ then $F(C) > 0$.
		\item \label{condition:resolving-homology:continuity} (continuity) If $\left(I_n\right)_n\subset \RR$ is an increasing chain of intervals in $\RR$ with limit $\bigcup_n I_n = I$ and $C$ is a simple subset of $X$ then $\lim_{n\to\infty} F(C\cap \pi^{-1}(I_n)) = F(C\cap \pi^{-1}(I))$.
	\end{enumerate}
	Then there is a positive function $\phi:X\to\left(0,\infty\right)$ such that for any simple subset $C\subset X$,
	\begin{equation}\label{eq:setfunc2pointfunc}
		F(C) = \sum_{x} \phi(x) d_C(x)
	\end{equation}
\end{lemma}
Note that the RHS of \eqref{eq:setfunc2pointfunc} is just a signed sum of $\phi$ over $\partial C$, similar to the fundamental theorem of calculus. Indeed, for $X=\RR$ and $C=\left[a,b\right]\subset\RR$, \eqref{eq:setfunc2pointfunc} reduces to $F(\left[a,b\right]) = \phi(b)-\phi(a)$.
\begin{proof}
	We may assume that $X$ is connected. As the intersection of two simple subsets is simple, we may think of $F$ as a map from simple functions to $\RR$, where we define a simple function $f:X\to\RR$ to be a function such that $f(X)$ is finite and for any $t\in\RR$, $f^{-1}(t)\subset X$ is simple, in the following way:
	\[
	F(f) = \sum_{t\in f(X)} tF(f^{-1}(t))
	\]
	
	The conditions on $F$ imply that $F$ is linear on simple functions: $F(f+g) = F(f) + F(g)$.
	
	Let $V_0$ be the free $\RR$-vector space on the points of $X$ and let $V_1$ be the vector space of simple functions $X\to\RR$. Let $d:V_1\to V_0$ be the map
	\[
	d(f) = \sum_{t\in f(X)}\sum_{x} t d_{f^{-1}(t)}(x) \left[x\right]
	\]
	where $\left[x\right]\in V_0$ is the generator of $V_0$ corresponding to $x\in X$.
	Then $d:V_1\to V_0$ is a homomorphism. Let $j:V_0\to H = V_0 / d(V_1)$ be its cokernel.
	Note that the conditions on $F$ imply that if $d(f)=0$ then $F(f) = 0$. Therefore $F$ factors through $d$: there is some linear functional $\widetilde F:V_0\to \RR$ such that $F = \widetilde F\circ d$. Note that \eqref{eq:setfunc2pointfunc} is equivalent to $F = \phi\circ d$, therefore if $\phi$ exists it must be the case that $\phi - \widetilde F$ factors through $j$: for some linear functional $g:H\to \RR$, $\phi - \widetilde F = g \circ j$.
	
	Therefore it is enough to find a linear functional $g:H\to\RR$ such that for any $x\in X$,
	\begin{equation}\label{eq:pointinequality}
		g\left(j\left(\left[x\right]\right)\right) > -\widetilde F\left(\left[x\right]\right) 
	\end{equation}
	This is effectively an infinite system of linear inequalities where the unknown vector is $g\in H^\ast$. Note that the positivity condition on $F$ implies that any finite subset of this system is solvable, therefore it is enough to show that there is a finite subset of this system which implies all of the inequalities in the system.
	
	Note that $H$ is finite-dimensional, and in fact the set $K=\left\{j\left(\left[x\right]\right):x\in X\right\}\subset H$ spans $H$ and is finite (because the map $x\mapsto j\left(\left[x\right]\right)$ is constant on each connected component of $X\setminus\sing{X}$).
	For each $h\in K$ let 
	\[
	X(h)=\left\{x\in X : j\left(\left[x\right]\right)=h\right\}
	\]
	Then \eqref{eq:pointinequality} becomes
	\begin{equation}\label{eq:pointinequality2}
		\forall h\in K, x\in X(h): g(h) > -\widetilde F\left(\left[x\right]\right)
	\end{equation}
	For $h\in K$ such that $\max_{x\in X(h)} -\widetilde F\left(\left[x\right]\right)$ exists (say with $x=x_0$), we are done, as the other instances of \eqref{eq:pointinequality2} involving $h$ are implied from the instance where $x=x_0$.
	
	For $h\in K$ such that $\max_{x\in X(h)} -\widetilde F\left(\left[x\right]\right)$ does not exist, let $x_n\in X(h)$ be a sequence of points such that
	\[
	\lim_{n\to\infty}-\widetilde F\left(\left[x_n\right]\right)
	=
	\sup_{x\in X(h)} -\widetilde F\left(\left[x\right]\right)
	.
	\]
	By passing to a subsequence we may assume that $x_n\notin \sing{X}$, that $\pi\left(x_n\right)$ is monotonic (either decreasing or increasing), and that all of $x_n$ are on the same connected component $R$ of $X\setminus\sing{X}$. Without loss of generality assume that $\left(\pi(x_n)\right)_n\subset\RR$ is increasing (the case where it is decreasing is similar). We have an increasing chain of open intervals $I_n = \left(\pi(x_1),\pi(x_n)\right)\subset \RR$ with limit $I=\left(\pi(x_1),\sup_n \pi(x_n)\right)$. Applying the continuity condition with the simple subset $R$ and the chain of intervals $I_n$, we get that
	\begin{multline*}
	\lim_{n\to\infty}\widetilde F\left(\left[x_n\right]\right)
	-
	\widetilde F\left(\left[x_1\right]\right)
	=
	\lim_{n\to\infty}F\left(R\cap\pi^{-1}(I_n)\right)
	=\\=
	F\left(R\cap\pi^{-1}(I)\right)
	=
	\sum_{\substack{y\in X
		\\ x_n\to y}} 
	\widetilde F\left(\left[y\right]\right)
	-
	\widetilde F\left(\left[x_1\right]\right)
	.
	\end{multline*}
	Therefore we have
	\begin{align*}
		-
		\sum_{\substack{y\in X
				\\ x_n\to y}} 
		\widetilde F\left(\left[y\right]\right)
		&=
		\sup_{x\in X(h)} -\widetilde F\left(\left[x\right]\right) \\
			\sum_{\substack{y\in X
				\\ x_n\to y}} 
		 j\left(\left[y\right]\right)
		&=
		h
		.
	\end{align*}
	As each $y\in X$ such that $x_n\to y$ must be in $\sing{X}$, we get that the instance of \eqref{eq:pointinequality2} for our choice of $h$ follows from the (finitely-many) instances of \eqref{eq:pointinequality2} for $x\in\sing{X}$.
	This completes the proof of lemma \ref{lemma:resolving-homology}.
\end{proof}

\section{Construction of metric - Proof of theorem \ref{thm:interpolation-metric}}
Let $M$ be a smooth orientable surface with a given smooth measure $\mu$, and $f : M\to\RR$ a smooth function without critical points. Suppose that there is an open set $U\subset M$ and a Riemannian metric $g_U$ on $U$ compatible with $\mu$ such that $M\setminus U$ is compact and $\left.f\right|_U$ is a Laplacian eigenfunction of $g_U$ with eigenvalue $-\lambda$. Suppose also that the given $\mu$ is admissible (as defined in the beginning of section \ref{section:interpolation-formulation}).

We will prove that there is a metric $g$ on $M$, equal to $g_U$ on some $U'\subset M$ with $M\setminus U'$ compact, for which $f$ is a Laplacian eigenfunction with eigenvalue $-\lambda$.

Given a smooth curve $\gamma\subset M$ such that $f\circ\gamma$ is constant, we denote by $\mu_f$ the measure on $\gamma$ induced from $\mu$ and $f$ in the following way: for any smooth $\phi:M\to\RR$ we have
\[
\int_\gamma \phi d\mu_f = \int_\gamma u_\phi\cdot d\hat{\mathbf{n}} 
,
\]
where $u_\phi$ is a vector field on a neighborhood of $\gamma$ for which $\partial_{u_\phi} f=\phi$. Note that this does not depends on the exact choice of $u_\phi$; and the RHS is defined as in section \ref{section:notation}. For example, given a coordinate system $x,y$ with $f(x,y)=\alpha y$ (where $\alpha$ is a constant), an arbitrary Riemannian metric $g$, a curve $\gamma(t)=\left(\gamma_x(t),0\right)$ with image $\left[0,L\right]\times\left\{0\right\}$, and a smooth function $\phi$, we have $u_\phi = \psi\partial_x + \alpha^{-1}\phi\partial_y$ for an arbitrary smooth function $\psi$ and so
\begin{multline*}
\int_\gamma \phi d\mu_f = \int_\gamma u_\phi\ddn = \int\left(\alpha^{-1}\phi(\gamma(t))\right)\gamma_x'(t)\sqrt{\left|g(\gamma(t))\right|}dt
=\\=
\int_0^L \phi(t,0)\alpha^{-1}\sqrt{\left|g(t,0)\right|}dt
\end{multline*}

Given two sets $R,U\subset M$ such that $\partial R\setminus U$ is composed of smooth curves on which $f$ is locally constant, we will use the notation
\[
\int_{\substack{\partial R\setminus {U}  \\ \text{(oriented)} }} \phi d\mu_f
\]
for integrating the function $\phi$ using the measure $\mu_f$ over the smooth curves of $\partial R\setminus U$ such that each curve $\gamma\subset \partial R\setminus U$ is oriented positively iff $df$ points on $\gamma$ to outside $R$. Continuing the example from the previous paragraph, if $U=\RR^2\setminus\left(\left[0,L\right]\times\RR\right)$ and $R=\RR\times\left[a,b\right]$ 
then
\[
\int_{\substack{\partial R\setminus {U}  \\ \text{(oriented)} }} \phi d\mu_f
=
\int_{\left[0,L\right]\times\left\{b\right\}} \phi d\mu_f
-
\int_{\left[0,L\right]\times\left\{a\right\}} \phi d\mu_f
\]

By lemma \ref{lemma:pointwise}, specifying a Riemannian metric $g$ on $M$ is the same as specifying the vector field $u = \nabla_g f$ which is defined by $u^i = g^{ij}\partial_j f$.
In order for a smooth vector field $u$ to correspond to a Riemannian metric which eventually agrees with $g_U$ and for which $f$ is a Laplacian eigenfunction, $u$ needs to satisfy the following conditions:
\begin{enumerate}
	\item $\partial_u f > 0$,
	\item $\nabla\cdot u = -\lambda f$,
	\item outside a compact set, $u$ agrees with $\nabla_{g_U} f$.
\end{enumerate}

First, we will find a function $h$ that would be $h=\partial_u f = \left\|\nabla_g f\right\|_g^2$. This function should satisfy the following conditions, for some $U'\subset U$ with $M\setminus U'$ compact:
\begin{enumerate}
	\item $h:M\to\left(0,\infty\right)$ is smooth and positive
	\item On $M\setminus U'$, $h$ agrees with $\left\|\nabla_{g_U} f\right\|_g^2$
	\item For any compact set $R\subset M$ for which $\partial R$ is a finite disjoint union of piecewise-smooth simple closed curve and $f$ is locally constant on $\partial R \setminus U'$, we have
	\[
	\int_{\partial R\cap U'} \left(\nabla_{g_U} f\right)\ddn
	+
	\int_{\substack{\partial R\setminus U' \\ \text{(oriented)}}} h d\mu_f
	=
	- \int_R \lambda f d\mu
	\]
\end{enumerate}

Let $U_1,U_2,U_3$ be open sets with $U_1 \subset \overline{U_1}\subset U_2 \subset \overline{U_2}\subset U_3 \subset \overline{U_3}\subset U$ and $M\setminus U_i$ compact. Assume also that $M\setminus U_1$ is a manifold with boundary; and that $\left. f\right|_{\partial U_1}$ has finitely many local extremal points.

For any $p\in M\setminus\overline{U_1}$, let $V_p\subset M\setminus\overline{U_1}$ be a neighborhood of $p$ such that the map ${\left. f\right|_{V_p}:V_p\to f(V_p)}$
factors as the composition of a diffeomorphism $V_p\xrightarrow{\sim}f(V_p)\times \left(0,1\right)$ and the projection ${f(V_p)\times\left(0,1\right)\to f(V_p)}$.
Then $\left(V_p\right)_{p\in M\setminus\overline{U_1}}$ is an open cover of $M\setminus\overline{U_1}$, so there is a finite subcover $\left(V_k\right)_{k\leq N}$ of $M\setminus U_2$. 
Let $\chi:M\to\left[0,1\right]$ and $\chi_k:M\to\left[0,1\right]$ for $k\leq N$ such that $\left\{\chi, \chi_k:k\leq N\right\}$ is a smooth partition of unity (i.e. $\chi+\sum_k \chi_k = 1 $), with $\chi_k^{-1}(0) = M\setminus V_k$ and $\chi(U_1) = 1$, $\chi(M\setminus U_3) = 0$.

We will have $h(x) = \chi(x)\left\|\nabla_{g_U} f(x)\right\|^2 + \phi(x)$ where $\phi:M\to\left[0,\infty\right)$ is a smooth nonnegative function supported in $M\setminus U_1$ and positive in $M\setminus U_3$. This $\phi$ should also satisfy the following condition: given a compact set $R\subset M$ for which $\partial R$ is a finite disjoint union of piecewise smooth simple closed curve and $f$ is locally constant on $\partial R \setminus \overline{U_1}$, we have
\begin{equation}\label{eq:phiintdef}
	\int_{\substack{\partial R\setminus {U_1}  \\ \text{(oriented)} }} \phi d\mu_f = 
	-\int_R \lambda fd\mu -
	\int_{
		\partial R
	}\left(\chi\nabla_{g_U} f\right)\ddn 
\end{equation}

Let $\pi: X\to\RR$ be the blueprint $(M\setminus\overline{U_1}) / \sim_f$ (see lemma \ref{lemma:quotient-is-blueprint}), and let $\xi:M\setminus\overline{U_1}\to X$ be the quotient map.
Given a simple subset $C\subset X$, define
\begin{multline}\label{eq:blueprint-chain}
	F(C) = - \int_R \lambda f d\mu
	- \int_{
		\partial R
	} \left(\chi\nabla_{g_U} f\right)\ddn
	= \\ =
	- \int_R \lambda f d\mu
	- \int_{
		\partial R\cap\overline{U_1}
	} \left(\chi\nabla_{g_U} f\right)\ddn
	-\int_{
		\substack{\partial R\setminus\overline{U_1}
			\\
			\text{(oriented)}
		}
	} \chi d\mu_f
\end{multline}
for any subset $R\subset M$ with piecewise-smooth boundary and which satisfy ${R\setminus U_1 = \xi^{-1}(C)}$. Note that $F(C)$ does not depends on the exact choice of $R$.

Note that $F$ satisfies the conditions of lemma \ref{lemma:resolving-homology} - conditions \ref{condition:resolving-homology:finite-to-zero}, \ref{condition:resolving-homology:additivity}  are obvious; conditions \ref{condition:resolving-homology:clopen-to-zero}, \ref{condition:resolving-homology:positivity} follows from $\mu$ being admissible; and condition \ref{condition:resolving-homology:continuity} follows from the continuity of the integrands in \eqref{eq:blueprint-chain}.
Therefore, there is a positive function $\phi_X:X\to\RR$ such that
\[
F(C) = \sum_{x\in\partial C} \phi_X(x) d_{C}(x)
,
\]
where $d_C(x)$ is the orientation of $\partial C$ at $x$ as defined in \eqref{eq:simple-set-boundary-orientation}.

Note that $\phi_X$ is pseudosmooth.
Indeed, let $J\subset \RR$ be a closed interval, let $W\subset X$ be a connected component of $\pi^{-1}(J)$, and let $R\subset M$ be a piecewise-smooth region with ${R\setminus U_1 = \xi^{-1}(W)}$. Then the function
\[
\widetilde{\phi}(t) = \sum_{x\in W\cap\pi^{-1}(t)} \phi_X(x)
\]
satisfy for any $t_1,t_2\in\RR$ with $\left[t_1,t_2\right]\subset J$
\begin{multline*}
	\widetilde{\phi}(t_2)-\widetilde{\phi}(t_1) = F(W\cap\pi^{-1}(\left[t_1,t_2\right]))
	=\\=
	-\int_{R\cap f^{-1}(\left[t_1,t_2\right])} \lambda fd\mu
	-\int_{\partial(R\cap f^{-1}(\left[t_1,t_2\right]))}\left(\chi\nabla_{g_U} f\right)\ddn
\end{multline*}
with the right hand side being smooth in $t_1,t_2$.

For $k\leq N$ let $I_k = \xi(V_k) \subset X$, and let $\widehat \chi_k :I_k\to\RR$ defined by
\[
\widehat\chi_k (x) = \int_{\xi^{-1}(x)} \chi_k d\mu_f
.
\]

Then by lemma \ref{lemma:function-decomposition} there are positive strictly-smooth functions $\phi_k:I_k\to\RR$ such that
\[
\sum_k \phi_k(x)\widehat\chi_k(x) = \phi_X(x)
\]
(in order to apply lemma \ref{lemma:function-decomposition} literally, one should need to restrict to the subblueprint ${X'=\xi\left(\bigcup_k V_k\right) = \bigcup_k I_k}$ and note that near the ends of $X'$ we have $\sum_k\widehat\chi_k = \phi_X$)

Then one can see that the following definition for $\phi:M\to\RR$, $h:M\to\RR$ works:
\[
\phi(p) = \sum_k \phi_k(\xi(p)) \chi_k(p)
\]
\[
h(p) = \chi(p)\left\|\nabla_{g_U} f(p)\right\|^2 + \sum_k \phi_k(\xi(p))\chi_k(p)
\]
Indeed, $h$ is positive and smooth on $M$ and $h=\left\|\nabla_{g_U} f\right\|^2$ on $U_1$. Additionally, for any piecewise-smooth compact region $R\subset M$ with $\left.f\right|_{\partial R\setminus U_1}$ locally constant we have
\begin{multline*}
	\int_{\substack{\partial R\setminus U_1\\ \text{(oriented)}}}h d\mu_f =\\=
	\int_{\substack{\partial R\setminus U_1\\ \text{(oriented)}}}\left(\chi\nabla_{g_U} f\right)\ddn + \int_{\substack{\partial R\setminus U_1\\ \text{(oriented)}}} \sum_k \phi_k(\xi(p))\chi_k(p) d\mu_f(p)
	=\\=
	\int_{\substack{\partial R\setminus U_1\\ \text{(oriented)}}}\left(\chi\nabla_{g_U} f\right)\ddn +
	\sum_{x\in\xi(\partial R\setminus U_1)}\sum_k \phi_k(x)\widehat{\chi}_k(x) d_{\xi(R\setminus U_1)}(x)
	=\\=
	\int_{\substack{\partial R\setminus U_1\\ \text{(oriented)}}}\left(\chi\nabla_{g_U} f\right)\ddn +
	\sum_{x\in\partial(\xi( R\setminus U_1))}\phi_X(x) d_{\xi(R\setminus U_1)}(x)
	=\\=
	\int_{\substack{\partial R\setminus U_1\\ \text{(oriented)}}}\left(\chi\nabla_{g_U} f\right)\ddn +
	F(\xi( R\setminus U_1))
	=\\=
	-\int_R \lambda f d\mu - \int_{\partial R\cap \overline{U_1}}\left(\nabla_{g_U} f\right)\ddn
	.
\end{multline*}
Therefore we get that for any vector field $u$ which agree on $U_1$ with $\nabla f$ and satisfy $\partial_u f = h$ and for any $R\subset M$ as above, we have
\begin{equation}\label{eq:contour-integral}
	\int_{\partial R} u\ddn = -\int_R \lambda fd\mu
	.
\end{equation}
The only thing remaining is to find such vector field $u$ such that \eqref{eq:contour-integral} would hold for any piecewise-smooth compact region $R\subset M$, even when $\left.f\right|_{\partial R\setminus U_1}$ is not locally constant.
The conditions on $u$ are:
\begin{enumerate}
	\item $\partial_u f = h$ on $M$
	\item $u$ agrees with $\nabla_{g_U} f$ on $U_1$ (or on arbitrary open subset of $U_1$ with compact complement with respect to $M$)
	\item $\nabla\cdot u = -\lambda f$. Equivalently, for any compact $R\subset M$ with piecewise smooth boundary,
	\[
	\int_{\partial R} u\ddn = -\int_R \lambda fd\mu
	\]
\end{enumerate}

In the following say that a smooth curve $\gamma$ is \emph{transverse to }$f$ if $\left(f\circ\gamma\right)'\neq 0$.
For any set $W\subset M$, define
\[
E(W) = \left\{p\in M:
\substack{
	\exists\text{a path }\gamma\text{ from }p\text{ to some }q\in W\text{ such that }f\circ\gamma\text{ is constant; and}
	\\
	\exists\text{a smooth curve }\gamma'\subset W\text{ passing through }q\text{ and transverse to }f
}
\right\}\subset M
\]
\begin{claim}
	\label{claim:sideway-integration}
	Given $h:M\to\RR$ and a definition of $u$ on some $W\subset M$, $u$ can be extended uniquely to $E(W)$, such that the above conditions on $u$ are satisfied.
\end{claim}
\begin{proof}[claim \ref{claim:sideway-integration}]
	Let $p\in E(W)$, we will show how to define $u(p)$.
	 Let $C \ni p$ be the connected component of the level set of $f$ containing $p$, and let $u_0$ be an arbitrary vector field in a neighborhood of $C$ such that $\partial_{u_0}f=h$ and $u_0$ agrees with $\nabla f$ on $U_1$. Then on $C$ the function $\nabla\cdot u_0+\lambda f$ is compactly-supported and by \eqref{eq:contour-integral} we have
	\begin{equation}\label{eq:global-integral-welldefined}
	\int_{\gamma_0} \left(\nabla \cdot u_0+\lambda f\right)d\mu_f = 0
	\end{equation}
	for any curve $\gamma_0$ in $C$ connecting two points in $U_1\cap C$.
	We claim that as a result of that, there is a smooth vector field $u_1$ on a neighborhood of $C$, and the restriction $\left. u_1\right|_C$ is unique, which satisfies the following conditions:
	\begin{enumerate}
		\item $u_1$ is tangent to the level lines of $f$;
		\item $u_1$ vanishes on $U_1\cap C$;
		\item The following equation holds on $C$:
		\begin{equation}\label{eq:to-be-integrated}
			\nabla\cdot u_1 = \nabla \cdot u_0+\lambda f
		\end{equation}
	\end{enumerate}
	We also claim that the restriction $\left. u_1 \right|_C$ is unique.
	
	Using this claim, we can define $u=u_0-u_1$ and it would be the unique choice for $u$ at $C$.
	
	To prove the uniqueness of $\left. u_1\right|_C$, let $q\in C$ be an arbitrary point and we will work with local coordinates $(x,y)$ near $q$, such that $y=f$. In such coordinates, $u_1$ at $(x,y)$ must be of the form $(v(x,y),0)$. Then \eqref{eq:to-be-integrated} has the form
	\begin{equation}
		\partial_x\left(\sqrt{\left|g\right|}v\right) = \sqrt{\left|g\right|}P
	\end{equation}
	where $P=\nabla\cdot u_0+\lambda f$ is the RHS of \eqref{eq:to-be-integrated}. In particular, if $q'\in C$ is another point which is close enough to $q$, we get, where $q=(x_q, a)$ and $q'=(x_{q'}, a)$:
	\begin{equation} \label{eq:local-integration}
		\left.\left(\sqrt{\left|g\right|}v\right)\right|_q^{q'}
		=
		\int_{x_q}^{x_{q'}} \left.\left( \sqrt{\left|g\right|}P\right)\right|_{(t,a)} dt
		=
		\int_{\left[q,q'\right]} P d\mu_f
	\end{equation}
	Note that the term $\sqrt{\left|g\right|}v$ from the RHS of \eqref{eq:local-integration} is the same in every coordinate chart $(x,y)$ in which $y=f$, and in fact it is the ratio between the functionals $\mu\left(u_1\wedge \cdot\right)$ and $df$. Therefore it follows that $u_1(q)$ can be calculated from the value of
	\begin{equation}\label{eq:global-integral}
	\int_{\gamma_q} \left(\nabla\cdot u_0 + \lambda f\right)d\mu_f
	\end{equation}
	where $\gamma_q$ is a curve in $C$ which connects a point in $U_1\cap C$ and $q$.
	This shows that $\left. u_1\right|_C$ is unique. Additionally, \eqref{eq:global-integral-welldefined} implies that \eqref{eq:global-integral} does not depend on the particular choice of $\gamma_q$, hence gives a definition of $u_1$.
\end{proof}

To finish the proof of theorem \ref{thm:interpolation-metric}, recall that $h$ and $\mu$ are fixed and we start with $u$ defined outside a compact set. If $u$ is defined on a set $W$ with $W \subsetneq E(W) \subset M$, then we can extend $u$ to $E(W)$ using claim \ref{claim:sideway-integration}. Otherwise, if $u$ is defined on a set $W$ with $W = E(W) \subsetneq M$, then we can choose a smooth path $\gamma$ between two connected components of $W$ such that $f\circ\gamma$ has no critical points, and smoothly extend $u$ from $W$ to $W\cup\gamma$ arbitrarily under the condition that $\partial_u f = h$. This gives an extension of $u$ to $W' = W\cup\gamma \supsetneq W$, and with claim \ref{claim:sideway-integration} we get an extension of $u$ to $E(W')$. Repeating this finitely many times we get an extension of $u$ to the whole $M$, as needed, completing the proof of theorem \ref{thm:interpolation-metric}.


\section{Proof of theorem \ref{thm:main-perturb-result}}
In this section we will prove theorem \ref{thm:main-perturb-result}.


\begin{definition}
	Let $X$ be an oval configuration in $\sph^2$. Then there is a two-coloring of the set of connected components of $\sph^2\setminus X$ such that for any two different connected components ${C_1, C_2\subset \sph^2\setminus X}$ whose boundary intersect have different colors; and this two-coloring is unique up to switching the colors. We will call it \emph{the canonical two-coloring} $c_X$ of $\sph^2\setminus X$.
\end{definition}
\begin{definition}
	Let $X$ be an oval configuration in $\sph^2$, and let $Y\subset X$ consists of a subset of the set of ovals from $X$. We will say that $X$ \emph{nicely-contains} $Y$ iff the canonical two-coloring $c_{X\setminus Y}$ of $\sph^2\setminus\left(X\setminus Y\right)$ satisfy the condition that all the ovals of $Y$ have the same color.
\end{definition}
Note that if $X$ is an oval configuration and $Y\subset X$ is a subset of the set of ovals of $X$, such that the canonical two-colorings $c_X, c_Y$ of $\sph^2\setminus X, \sph^2\setminus Y$ respectively agree on the neighborhood of each oval from $Y$, then $X$ nicely-contains $Y$. Indeed, we have that $c_{X\setminus Y}$ is the XOR of $c_X$ and $c_Y$ on $\sph^2\setminus X$, so in particular if $c_X$ and $c_Y$ agree near $Y$ then $c_{X\setminus Y}$ has only one color near $Y$.

\begin{figure}
	\centering
	\begin{tikzpicture}[x=0.75pt,y=0.75pt,yscale=-1,xscale=1]
		\draw  [fill={rgb, 255:red, 0; green, 0; blue, 0 }  ,fill opacity=1 ] (35.84,91.62) .. controls (35.84,89.34) and (37.69,87.49) .. (39.97,87.49) .. controls (42.25,87.49) and (44.1,89.34) .. (44.1,91.62) .. controls (44.1,93.9) and (42.25,95.75) .. (39.97,95.75) .. controls (37.69,95.75) and (35.84,93.9) .. (35.84,91.62) -- cycle ;
		\draw    (44.1,91.62) -- (74.18,91.62) ;
		\draw    (9.18,91.62) -- (35.84,91.62) ;
		\draw    (39.97,95.75) -- (39.97,121.62) ;
		\draw    (39.97,69.82) -- (39.97,87.49) ;
		\draw  [draw opacity=0] (223.51,68.14) .. controls (223.51,56.08) and (233.37,46.3) .. (245.53,46.3) -- (245.53,68.14) -- cycle ; \draw   (223.51,68.14) .. controls (223.51,56.08) and (233.37,46.3) .. (245.53,46.3) ;  
		\draw  [draw opacity=0] (223.51,26) .. controls (223.51,26) and (223.51,26) .. (223.51,26) .. controls (223.51,37.4) and (214.27,46.64) .. (202.87,46.64) -- (202.87,26) -- cycle ; \draw   (223.51,26) .. controls (223.51,26) and (223.51,26) .. (223.51,26) .. controls (223.51,37.4) and (214.27,46.64) .. (202.87,46.64) ;  
		\draw    (245.53,46.3) -- (263.18,46.3) ;
		\draw    (223.51,68.14) -- (223.51,83.82) ;
		\draw    (186.18,46.64) -- (202.87,46.64) ;
		\draw    (223.51,6.82) -- (223.51,26) ;
		\draw  [draw opacity=0] (244.15,146.36) .. controls (244.15,146.36) and (244.15,146.36) .. (244.15,146.36) .. controls (232.75,146.36) and (223.51,137.12) .. (223.51,125.72) -- (244.15,125.72) -- cycle ; \draw   (244.15,146.36) .. controls (244.15,146.36) and (244.15,146.36) .. (244.15,146.36) .. controls (232.75,146.36) and (223.51,137.12) .. (223.51,125.72) ;  
		\draw  [draw opacity=0] (202.43,145.11) .. controls (202.43,145.11) and (202.43,145.11) .. (202.43,145.11) .. controls (202.43,145.11) and (202.43,145.11) .. (202.43,145.11) .. controls (213.83,145.11) and (223.07,154.35) .. (223.07,165.75) -- (202.43,165.75) -- cycle ; \draw   (202.43,145.11) .. controls (202.43,145.11) and (202.43,145.11) .. (202.43,145.11) .. controls (202.43,145.11) and (202.43,145.11) .. (202.43,145.11) .. controls (213.83,145.11) and (223.07,154.35) .. (223.07,165.75) ;  
		\draw    (186.18,145.11) -- (202.43,145.11) ;
		\draw    (244.15,146.36) -- (263.18,146.36) ;
		\draw    (223.07,181.82) -- (223.07,165.75) ;
		\draw    (223.51,105.82) -- (223.51,125.72) ;
		\draw [line width=2.25]    (91.84,93) .. controls (137.45,94.68) and (141.03,86.26) .. (174.54,58.95) ;
		\draw [shift={(177.18,56.82)}, rotate = 141.15] [color={rgb, 255:red, 0; green, 0; blue, 0 }  ][line width=2.25]    (24.48,-10.98) .. controls (15.57,-5.15) and (7.41,-1.49) .. (0,0) .. controls (7.41,1.49) and (15.57,5.15) .. (24.48,10.98)   ;
		\draw [line width=2.25]    (91.84,93) .. controls (133.88,93.33) and (144.55,106.02) .. (178.92,124.12) ;
		\draw [shift={(182.18,125.82)}, rotate = 207.18] [color={rgb, 255:red, 0; green, 0; blue, 0 }  ][line width=2.25]    (24.48,-10.98) .. controls (15.57,-5.15) and (7.41,-1.49) .. (0,0) .. controls (7.41,1.49) and (15.57,5.15) .. (24.48,10.98)   ;
		
		\draw (216,90) node [anchor=north west][inner sep=0.75pt]   [align=left] {or};
	\end{tikzpicture}
	\caption{Perturbation of a vertex}
	\label{fig:vertex-perturbation}
\end{figure}
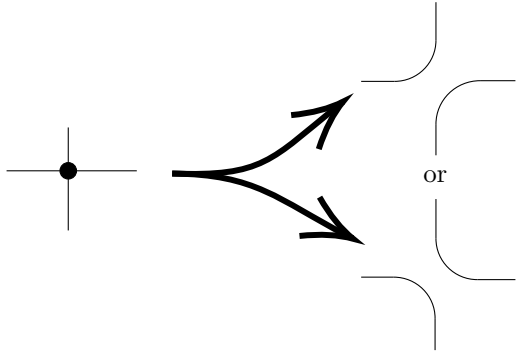

\begin{definition}
	Let $G$ be an embedded graph in $\sph^2$ which is connected and $4$-regular.
	For each vertex $v$ of $G$, let $e_1,e_2,e_3,e_4$ be the edges incident to $v$ ordered cyclically, and choose a pairing $P_v$ of those edges 
	from one of $\left\{\left(e_1,e_2\right),\left(e_3,e_4\right)\right\}$, $ \left\{\left(e_1,e_4\right),\left(e_2,e_3\right)\right\}$.
	Remove $v$ and connect each pair of edges in $P_v$ together (see figure \ref{fig:vertex-perturbation}).
	Doing this for each vertex $v$ of $G$ one gets an oval configuration $X$ in $\sph^2$. We will say that an oval configuration $X$ is a \emph{perturbation} of $G$ if it is a result of the above process (for some choice of $P_v$ for each $v$).
\end{definition}

\begin{lemma} \label{lemma:reduction}
	Let $G$ be an embedded graph in $\sph^2$ which is connected and $4$-regular. Let $X$ be an oval configuration which is a perturbation of $G$. Let $Y\subset X$ be an oval configuration nicely-contained in $X$. Then $Y$ is also a perturbation of $G$.
\end{lemma}
\begin{proof}[lemma \ref{lemma:reduction}]
	We will prove by induction on $\left|X\setminus Y\right|$. The case $\left|X\setminus Y\right|=0$ is trivial, so assume $\left|X\setminus Y\right|\geq 1$.
	
	In the following, "face" means a face of the graph $G$ while "region" is an area bounded by some collection of ovals (which is a perturbation of $G$). Note that, given a perturbation of $G$, any face belongs to a unique region while a region may consist of several faces of $G$. 
	
	\emph{Case 1:} There is a region $R$ of $\sph^2\setminus X$ such that $\partial R$ consists of only ovals from $X\setminus Y$, and there are at least $2$ such ovals in $\partial R$ (see figure \ref{fig:reduction-case-1}).
	
	In this case, because the graph $G$ is connected, it follows that there is a face $F$ of $G$ which is inside $R$, with two edges $e_1,e_2$ in $\partial F$ which are neighbors (in the cyclic order of $\partial F$) and belong to different ovals (those ovals are necessarily from $X\setminus Y$). Let $v$ be the common vertex of $e_1,e_2$.
	As $e_1,e_2$ belong to different ovals, they are not paired with each other in $P_v$. Let $e_3$ be the edge paired to $e_2$ in $P_v$ and let $e_4$ be the edge paired to $e_1$ in $P_v$. Then $e_1, e_2, e_3, e_4$ is their cyclic order as the edges incident to $v$.
	We can change the choice of pairing $P_v$ to the other choice $P'_v=\left\{\left(e_1,e_2\right),\left(e_3,e_4\right)\right\}$.
	This gives rise to a new perturbation $X'$ of $G$ which has one less oval than $X$ (because now $e_1,e_2,e_3,e_4$ belong to the same oval and we did not change any other oval) and still nicely includes $Y$.
	
	\emph{Case 2:} For any region $R$ of $\sph^2\setminus X$, either $\partial R$ contains an oval from $Y$, or $\partial R$ consists of exactly $1$ oval from $X\setminus Y$.
	
	Note that from the assumption that $X$ nicely-contains $Y$ it follows that each oval $C$ of $X\setminus Y$ borders at least one region of $\sph^2\setminus X$ whose boundary consists only of ovals from $X\setminus Y$ (specifically it is the region which is colored by $c_{X\setminus Y}$ in the other color than $Y$); and by the case description, this region has only one boundary component, which is $C$. It follows that for each oval $C$ from $X\setminus Y$, all the other ovals from $X$ are on the same side with respect to $C$.
	
	Because $G$ is connected, there is a face $F$ of $G$ and two edges $e_1,e_2$ in $\partial F$ which are neighbors (in the cyclic order of $\partial F$), where $e_1$ belongs to an oval $C_1$ from $Y$ and $e_2$ belongs to an oval $C_2$ from $X\setminus Y$
	(see figure \ref{fig:reduction-case-2}). Let $v$ be the common vertex of $e_1,e_2$. Let $e_3$ be the edge paired to $e_2$ in $P_v$ and let $e_4$ the edge paired to $e_1$ in $P_v$. Then $e_1, e_2, e_3, e_4$ is their cyclic order as the edges incident to $v$.
	We can change the choice of pairing $P_v$ to the other choice $P'_v=\left\{\left(e_1,e_2\right),\left(e_3,e_4\right)\right\}$.
	This gives rise to a new perturbation $X'$ of $G$ which has one less oval than $X$ (because now $e_1,e_2,e_3,e_4$ belong to the same oval $C'$ and we did not change any other oval).
	Let $Y'$ be $\left(Y\setminus\left\{C_1\right\}\right)\cup\left\{C'\right\}$. Then $X'$ includes $Y'$ which is equivalent to $Y$: Indeed from this case description it follows that the side of $C_2$ which does not include $F$ does not have any ovals in it, therefore moving from $X$ to $X'$ did not change the containment configuration of ovals from $X\setminus \left\{C_2\right\}$ (with $C'$ in $X'$ in place of $C_1$ in $X$).
\end{proof}

\begin{figure}
	\centering
	\begin{tikzpicture}[x=0.75pt,y=0.75pt,yscale=-1,xscale=1]
		\clip (0,0) rectangle (480,200);
		\draw  [dash pattern={on 0.84pt off 2.51pt}] (120.55,125.83) .. controls (129.33,134.62) and (129.33,148.87) .. (120.55,157.66) .. controls (111.76,166.45) and (97.51,166.45) .. (88.72,157.66) .. controls (79.93,148.87) and (79.93,134.62) .. (88.72,125.83) .. controls (97.51,117.04) and (111.76,117.04) .. (120.55,125.83) -- cycle ;
		\draw  [draw opacity=0] (89.86,157.66) .. controls (93.79,153.73) and (99.21,151.3) .. (105.2,151.3) .. controls (111.2,151.3) and (116.62,153.73) .. (120.55,157.66) -- (105.2,173) -- cycle ; \draw   (89.86,157.66) .. controls (93.79,153.73) and (99.21,151.3) .. (105.2,151.3) .. controls (111.2,151.3) and (116.62,153.73) .. (120.55,157.66) ;  
		\draw  [draw opacity=0] (119.4,125.83) .. controls (115.47,129.75) and (110.05,132.18) .. (104.06,132.18) .. controls (98.07,132.18) and (92.64,129.75) .. (88.72,125.83) -- (104.06,110.49) -- cycle ; \draw   (119.4,125.83) .. controls (115.47,129.75) and (110.05,132.18) .. (104.06,132.18) .. controls (98.07,132.18) and (92.64,129.75) .. (88.72,125.83) ;  
		\draw    (73.77,112.02) -- (88.72,125.83) ;
		\draw    (120.55,125.83) -- (135.82,110.55) ;
		\draw    (71.84,174.54) -- (88.72,157.66) ;
		\draw    (120.55,157.66) -- (137.43,174.54) ;
		\draw [line width=0.75]  [dash pattern={on 4.5pt off 4.5pt}]  (73.77,112.02) .. controls (6.89,41.75) and (194.66,40.46) .. (135.82,110.55) ;
		\draw  [dash pattern={on 4.5pt off 4.5pt}]  (71.84,174.54) .. controls (28.11,212.16) and (-58.06,8.31) .. (101.42,7.67) .. controls (260.89,7.03) and (180.51,223.09) .. (137.43,174.54) ;
		\draw  [dash pattern={on 0.84pt off 2.51pt}] (368.42,130.66) .. controls (376.83,122.25) and (390.46,122.25) .. (398.86,130.66) .. controls (407.27,139.07) and (407.27,152.69) .. (398.86,161.1) .. controls (390.46,169.51) and (376.83,169.51) .. (368.42,161.1) .. controls (360.02,152.69) and (360.02,139.07) .. (368.42,130.66) -- cycle ;
		\draw  [draw opacity=0] (398.86,160) .. controls (395.11,156.25) and (392.79,151.06) .. (392.79,145.33) .. controls (392.79,139.6) and (395.11,134.41) .. (398.86,130.66) -- (413.54,145.33) -- cycle ; \draw   (398.86,160) .. controls (395.11,156.25) and (392.79,151.06) .. (392.79,145.33) .. controls (392.79,139.6) and (395.11,134.41) .. (398.86,130.66) ;  
		\draw  [draw opacity=0] (368.42,131.76) .. controls (372.18,135.51) and (374.5,140.7) .. (374.5,146.43) .. controls (374.5,152.16) and (372.18,157.35) .. (368.42,161.1) -- (353.75,146.43) -- cycle ; \draw   (368.42,131.76) .. controls (372.18,135.51) and (374.5,140.7) .. (374.5,146.43) .. controls (374.5,152.16) and (372.18,157.35) .. (368.42,161.1) ;  
		\draw    (354.12,117.46) -- (368.42,131.76) ;
		\draw    (398.86,130.66) -- (413.47,116.05) ;
		\draw    (352.28,177.25) -- (368.42,161.1) ;
		\draw    (398.86,161.1) -- (415.01,177.25) ;
		\draw [line width=0.75]  [dash pattern={on 4.5pt off 4.5pt}]  (354.12,117.46) .. controls (290.16,50.25) and (469.74,49.02) .. (413.47,116.05) ;
		\draw  [dash pattern={on 4.5pt off 4.5pt}]  (352.28,177.25) .. controls (310.46,213.22) and (223.74,17.04) .. (380.57,17.65) .. controls (537.39,18.27) and (456.21,223.68) .. (415.01,177.25) ;
		\draw [line width=2.25]    (211,133) -- (270,133) ;
		\draw [shift={(274,133)}, rotate = 180] [color={rgb, 255:red, 0; green, 0; blue, 0 }  ][line width=2.25]    (17.49,-7.84) .. controls (11.12,-3.68) and (5.29,-1.07) .. (0,0) .. controls (5.29,1.07) and (11.12,3.68) .. (17.49,7.84)   ;
		
		\draw (159,98.4) node [anchor=north west][inner sep=0.75pt]    {$R$};
		\draw (127.33,120) node [anchor=north west][inner sep=0.75pt]    {$e_{1}$};
		\draw (136,160) node [anchor=north west][inner sep=0.75pt]    {$e_{2}$};
		\draw (61.33,158) node [anchor=north west][inner sep=0.75pt]    {$e_{3}$};
		\draw (62.33,116.07) node [anchor=north west][inner sep=0.75pt]    {$e_{4}$};
		\draw (408.33,123) node [anchor=north west][inner sep=0.75pt]    {$e_{1}$};
		\draw (413.33,161.73) node [anchor=north west][inner sep=0.75pt]    {$e_{2}$};
		\draw (339,160) node [anchor=north west][inner sep=0.75pt]    {$e_{3}$};
		\draw (342,119) node [anchor=north west][inner sep=0.75pt]    {$e_{4}$};
		\draw (99,135.73) node [anchor=north west][inner sep=0.75pt]    {$v$};
		\draw (376.33,139.07) node [anchor=north west][inner sep=0.75pt]    {$v$};
	\end{tikzpicture}
	\caption{Case 1 of proof of lemma \ref{lemma:reduction}}
	\label{fig:reduction-case-1}
\end{figure}
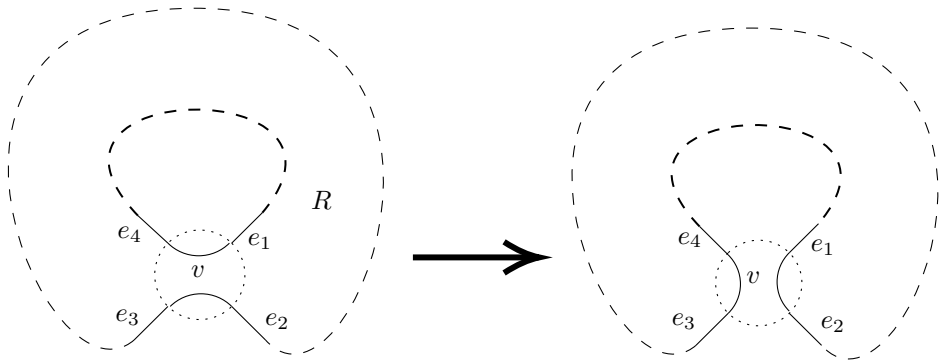

\begin{figure}
	\centering
	\begin{tikzpicture}[x=0.75pt,y=0.75pt,yscale=-1,xscale=1]
		\clip (0,0) rectangle (480,200);
		\draw  [dash pattern={on 0.84pt off 2.51pt}] (120.55,125.83) .. controls (129.33,134.62) and (129.33,148.87) .. (120.55,157.66) .. controls (111.76,166.45) and (97.51,166.45) .. (88.72,157.66) .. controls (79.93,148.87) and (79.93,134.62) .. (88.72,125.83) .. controls (97.51,117.04) and (111.76,117.04) .. (120.55,125.83) -- cycle ;
		\draw  [draw opacity=0] (89.86,157.66) .. controls (93.79,153.73) and (99.21,151.3) .. (105.2,151.3) .. controls (111.2,151.3) and (116.62,153.73) .. (120.55,157.66) -- (105.2,173) -- cycle ; \draw   (89.86,157.66) .. controls (93.79,153.73) and (99.21,151.3) .. (105.2,151.3) .. controls (111.2,151.3) and (116.62,153.73) .. (120.55,157.66) ;  
		\draw  [draw opacity=0] (119.4,125.83) .. controls (115.47,129.75) and (110.05,132.18) .. (104.06,132.18) .. controls (98.07,132.18) and (92.64,129.75) .. (88.72,125.83) -- (104.06,110.49) -- cycle ; \draw   (119.4,125.83) .. controls (115.47,129.75) and (110.05,132.18) .. (104.06,132.18) .. controls (98.07,132.18) and (92.64,129.75) .. (88.72,125.83) ;  
		\draw    (73.77,112.02) -- (88.72,125.83) ;
		\draw    (120.55,125.83) -- (135.82,110.55) ;
		\draw    (71.84,174.54) -- (88.72,157.66) ;
		\draw    (120.55,157.66) -- (137.43,174.54) ;
		\draw [line width=0.75]  [dash pattern={on 4.5pt off 4.5pt}]  (73.77,112.02) .. controls (6.89,41.75) and (194.66,40.46) .. (135.82,110.55) ;
		\draw  [dash pattern={on 4.5pt off 4.5pt}]  (71.84,174.54) .. controls (28.11,212.16) and (-58.06,8.31) .. (101.42,7.67) .. controls (260.89,7.03) and (180.51,223.09) .. (137.43,174.54) ;
		\draw  [dash pattern={on 0.84pt off 2.51pt}] (368.42,130.66) .. controls (376.83,122.25) and (390.46,122.25) .. (398.86,130.66) .. controls (407.27,139.07) and (407.27,152.69) .. (398.86,161.1) .. controls (390.46,169.51) and (376.83,169.51) .. (368.42,161.1) .. controls (360.02,152.69) and (360.02,139.07) .. (368.42,130.66) -- cycle ;
		\draw  [draw opacity=0] (398.86,160) .. controls (395.11,156.25) and (392.79,151.06) .. (392.79,145.33) .. controls (392.79,139.6) and (395.11,134.41) .. (398.86,130.66) -- (413.54,145.33) -- cycle ; \draw   (398.86,160) .. controls (395.11,156.25) and (392.79,151.06) .. (392.79,145.33) .. controls (392.79,139.6) and (395.11,134.41) .. (398.86,130.66) ;  
		\draw  [draw opacity=0] (368.42,131.76) .. controls (372.18,135.51) and (374.5,140.7) .. (374.5,146.43) .. controls (374.5,152.16) and (372.18,157.35) .. (368.42,161.1) -- (353.75,146.43) -- cycle ; \draw   (368.42,131.76) .. controls (372.18,135.51) and (374.5,140.7) .. (374.5,146.43) .. controls (374.5,152.16) and (372.18,157.35) .. (368.42,161.1) ;  
		\draw    (354.12,117.46) -- (368.42,131.76) ;
		\draw    (398.86,130.66) -- (413.47,116.05) ;
		\draw    (352.28,177.25) -- (368.42,161.1) ;
		\draw    (398.86,161.1) -- (415.01,177.25) ;
		\draw [line width=0.75]  [dash pattern={on 4.5pt off 4.5pt}]  (354.12,117.46) .. controls (290.16,50.25) and (469.74,49.02) .. (413.47,116.05) ;
		\draw  [dash pattern={on 4.5pt off 4.5pt}]  (352.28,177.25) .. controls (310.46,213.22) and (223.74,17.04) .. (380.57,17.65) .. controls (537.39,18.27) and (456.21,223.68) .. (415.01,177.25) ;
		\draw [line width=2.25]    (211,133) -- (270,133) ;
		\draw [shift={(274,133)}, rotate = 180] [color={rgb, 255:red, 0; green, 0; blue, 0 }  ][line width=2.25]    (17.49,-7.84) .. controls (11.12,-3.68) and (5.29,-1.07) .. (0,0) .. controls (5.29,1.07) and (11.12,3.68) .. (17.49,7.84)   ;
		
		\draw (100,65) node [anchor=north west][inner sep=0.75pt]    {$C_1$};
		\draw (172,98.4) node [anchor=north west][inner sep=0.75pt]    {$C_2$};
		\draw (127.33,120) node [anchor=north west][inner sep=0.75pt]    {$e_{1}$};
		\draw (136,160) node [anchor=north west][inner sep=0.75pt]    {$e_{2}$};
		\draw (61.33,158) node [anchor=north west][inner sep=0.75pt]    {$e_{3}$};
		\draw (62.33,116.07) node [anchor=north west][inner sep=0.75pt]    {$e_{4}$};
		\draw (408.33,123) node [anchor=north west][inner sep=0.75pt]    {$e_{1}$};
		\draw (413.33,161.73) node [anchor=north west][inner sep=0.75pt]    {$e_{2}$};
		\draw (339,160) node [anchor=north west][inner sep=0.75pt]    {$e_{3}$};
		\draw (342,119) node [anchor=north west][inner sep=0.75pt]    {$e_{4}$};
		\draw (99,135.73) node [anchor=north west][inner sep=0.75pt]    {$v$};
		\draw (376.33,139.07) node [anchor=north west][inner sep=0.75pt]    {$v$};
	\end{tikzpicture}
	\caption{Case 2 of proof of lemma \ref{lemma:reduction}}
	\label{fig:reduction-case-2}
\end{figure}

\begin{figure}
	\centering
	\begin{tikzpicture}[x=0.75pt,y=0.75pt,yscale=-1,xscale=1]
		\node[inner sep=0pt] (imgbefore) at (100,90) {\includegraphics[width=3cm]{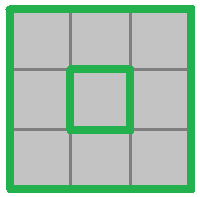}};
		\draw    (193,89) -- (358.45,89) ;
		\draw [shift={(360.45,89)}, rotate = 180] [color={rgb, 255:red, 0; green, 0; blue, 0 }  ][line width=0.75]    (15.3,-6.86) .. controls (9.73,-3.22) and (4.63,-0.93) .. (0,0) .. controls (4.63,0.93) and (9.73,3.22) .. (15.3,6.86)   ;
		\node[inner sep=0pt] (imgafter) at (450,90) {\includegraphics[width=3cm]{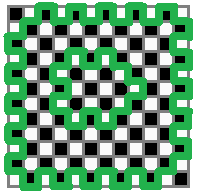}};
	\end{tikzpicture}
	\caption{Example of lemma \ref{lemma:chessboard-zigzag}}
	\label{fig:square-wave-pattern}
\end{figure}
\begin{lemma}
	\label{lemma:chessboard-zigzag}
	There is an absolute constant $M$ such that the following holds:
	Let $X$ be an oval configuration which can be drawn (recall definition \ref{def:can-be-drawn}) in the $n\times n$ grid graph, and fix a canonical two-coloring $c_X$.
	On the $Mn\times Mn$ grid graph, let $c_{Mn}$ be the $2$-coloring of its finite faces in chessboard pattern.
	Then there is a drawing $G'$ of $X$ in the $Mn\times Mn$ grid graph such that for any edge $e$ of $G'$, the black side of $e$ according to $c_X$ and the black side of $e$ according to $c_{Mn}$ agree (see figure \ref{fig:square-wave-pattern})
\end{lemma}
\begin{proof}[lemma \ref{lemma:chessboard-zigzag}]
	Let $M\geq 4$, let $h_1$ be the embedding of $X$ given by a drawing of $X$ in the $n\times n$ grid graph, and let $h_2$ be the obvious embedding (as a topological minor - for a definition of the term see for example section 1.7 of \cite{graph-theory}) of the $n\times n$ grid graph in the $(Mn+2)\times (Mn+2)$ grid graph (will be denoted by $G_{Mn+2}$) given by subdividing each face of the $n\times n$ grid into $M\times M$ grid and adding border of width $1$ around the whole grid.
	Then $h:=h_1\circ h_2$ is an embedding of $X$ as a drawing in $G_{Mn+2}$.
	Let $c_{Mn+2}$ be the $2$-coloring of the finite faces of $G_{Mn+2}$ given by chessboard pattern.
	
	Let $B_0$ be the set of faces of $G_{Mn+2}$ which belong under $h$ to a region which is colored black by $c_X$.
	
	Let $B_1$ be the set of $c_{Mn+2}$-black faces of $G_{Mn+2}$ which are adjacent to some face in $B_0$.
	
	Let $B_2$ be the set of $c_{Mn+2}$-white faces $F$ of $G_{Mn+2}$ such that all of the adjacent finite faces of $F$ are in $B_0\cup B_1$.
	
	Let $B=B_0\cup B_1\cup B_2$.
	
	Then it is easy to see that $\partial B$ is a drawing of $X$ in the $(Mn+2)\times (Mn+2)$ grid graph with the required condition.
\end{proof}

For $0\leq m\leq n$, let $Y_n^m$ be, maybe up to some multiplicative constant which is uninteresting for us, the spherical harmonic of degree $n$ which is of the form
\[
Y_n^m(\theta,\phi) = \sin^m \theta F_n^m(\cos \theta) \sin\left(m\phi\right)
\]
where $F_n^m$ is defined using the Legendre polynomial $P_n$ of degree $n$ by
\[
c_{n,m}\left(1-x^2\right)^{\frac m2}F_n^m(x) = \frac{d^m}{dx^m}P_n(x)
\]

\begin{lemma}\label{lemma:from-drawing-to-coarse-perturb}
	There is an absolute constant $D$ such that the following holds: let $X$ be an oval configuration which can be drawn in the $n\times n$ grid graph. Let $G$ be the zero set of the spherical harmonic $Y_{\left\lfloor 2Dn\right\rfloor}^{\left\lfloor Dn\right\rfloor}$, thought of as an embedded graph. Let $G'$ be a $4$-regular embedded graph obtained from $G$ by, for $v$ being each one out of the north pole and the south pole, removing $v$ and joining the edges 
	$\left\{e_1,e_2,\ldots,e_{2k}\right\}$
	incident to $v$ by the pairing $\left(e_1,e_2\right),\left(e_3,e_4\right),\ldots,\left(e_{2k-1},e_{2k}\right)$.
	
	Then there is a perturbation of $G'$ which nicely contains $X$.
\end{lemma}
\begin{proof}[lemma \ref{lemma:from-drawing-to-coarse-perturb}]
	Let $M$ as in lemma \ref{lemma:chessboard-zigzag} and let $H$ be the $Mn\times Mn$ grid graph.
	Choose $D$ such that $H$ is an induced subgraph of $G'$ (that is, $H$ is isomorphic to a graph that can be formed from $G'$ by removing some vertices, while keeping all the edges that connect between non-removed vertices).
	Draw $X$ in $H$ as in lemma \ref{lemma:chessboard-zigzag} and compose it with the inclusion $H\to G'$.
	Then we can choose a perturbation of $G'$ such that for any vertex $v$ of $G'$ on which $X$ passes through we choose the pairing $P_v$ that does not separate the two edges belonging to $X$.
	For any such perturbation we get that it nicely contains $X$.
\end{proof}

\begin{lemma}\label{lemma:metric-perturb}
	Let $f=Y_{\left\lfloor 2Dn\right\rfloor}^{\left\lfloor Dn\right\rfloor}$ be a middle-zonal spherical harmonic with eigenvalue (of the minus Laplacian) $\lambda$, let $G'$ as in lemma \ref{lemma:from-drawing-to-coarse-perturb}, and let $X$ be an oval configuration which is a perturbation of $G'$. Then there is an infinitesimal perturbation $g$ of the round metric of $\sph^2$ and an eigenfunction of the minus $g$-Laplacian with eigenvalue $\lambda$ and zero set equivalent to $X$.
\end{lemma}
\begin{proof}[lemma \ref{lemma:metric-perturb}]
	Let $S_1\subset\sph^2$ be the set of critical points of $f$ which are also zeroes of $f$, and let $S_2\subset\sph^2$ be the set of critical points of $f$ which are not zeroes of $f$.
	For each $p\in S_1\cup S_2$, let $U_p,V_p$ be open neighborhoods of $p$ such that:
	\begin{itemize}
		\item $U_p\subset \overline{U_p}\subset V_p$
		\item $V_p\cap V_q=\emptyset$ for $p\neq q \in S_1\cup S_2$
		\item $U_p$ is small enough so that there is a spherical harmonic (with eigenvalue $\lambda$) which is positive on $U_p$.
	\end{itemize}
	For each $p\in S_1$, let $\phi_p$ be a smooth bump function, supported at $V_p$, and identically equals on $U_p$ to some spherical harmonic (with eigenvalue $\lambda$ of the minus Laplacian) which is positive at $p$.
	Let $W = \sph^2\setminus\bigcup_p \overline{U_p}$, and let $\psi$ be a nonnegative smooth bump function, supported at $W$.
	For any choice of signs $s:S_1\to \left\{1,-1\right\}$, let $m_s\in\RR$ be a number which satisfy the following:
	\begin{equation}\label{eq:vanishing-integral}
		\int_{\sph^2}\left(m_s \psi + \sum_{p\in S_1} s(p)\phi_p\right) = 0
	\end{equation}
	Define, for $x\in\sph^2$ and $t\in\RR$,
	\[
	f_t(x) = f(x) + t\left(m_s\psi(x) + \sum_{p\in S_1}s(p)\phi_p(x)\right)
	\]
	Note that for each choice of $X$ as a perturbation of $G'$ there is a choice of $s$ and $\epsilon>0$ such that the zero set of $f_t$ for any $0<t<\epsilon$ is equivalent to $X$.
	
	Since $\int_{\sph^2}\Delta\phi_p=0$, \eqref{eq:vanishing-integral} is equivalent to
	\begin{equation}\label{eq:vanishing-integral-2}
		\int_{\sph^2}\left(m_s \psi + \sum_{p\in S_1} s(p)\left(\phi_p + \frac 1\lambda \Delta\phi_p\right)\right)=0
	\end{equation}
	Note that in \eqref{eq:vanishing-integral-2} the integrand is supported on $W$, so in particular compactly supported on $\sph^2\setminus\left(S_1\cup S_2\right)$.
	As $\sph^2\setminus\left(S_1\cup S_2\right)$ is connected, by corollary 5.8 in \cite{BottTu}, we have that the de-Rham cohomology with compact support $H^2_c\left(\sph^2\setminus\left(S_1\cup S_2\right)\right)$ is isomorphic to $\RR$; in other words every compactly supported $2$-form on $\sph^2\setminus\left(S_1\cup S_2\right)$ with vanishing integral is of the form $d\alpha$ for some compactly supported $1$-form $\alpha$ on $\sph^2\setminus\left(S_1\cup S_2\right)$.
	In particular, there is a smooth vector field $u_0$ on $\sph^2$ supported on some compact $W'$ with $W\subset W'\subset \sph^2\setminus\left(S_1\cup S_2\right)$ such that
	\[
	\nabla\cdot u_0 = m_s \psi + \sum_{p\in S_1}s(p)\left(\phi_p+\frac 1\lambda \Delta \phi_p\right)
	\]
	Let $u$ be the vector field
	\[
	u = -\lambda u_0 + \sum_{p\in S_1} s(p)\nabla \phi_p
	\]
	Then we get that, for any $t$,
	\[
	\nabla\cdot\left(\nabla f + tu\right) =-\lambda f_t
	\]
	and outside $W'$ we have $\nabla f + tu = \nabla f_t$.
	
	Note that for $t>0$ small enough, $\left\langle \nabla f+tu, \nabla f_t\right\rangle >0$ at any point where $\nabla f_t\neq 0$, and any point with $\nabla f_t=0$ is in $\sph^2\setminus W'$.
	Therefore, by lemma \ref{lemma:pointwise}, for any such $t$ there is a (unique) smooth metric $g_t$ such that $g_t$ induces the usual area measure and $\nabla_{g_t} f_t = \nabla f+tu$ (note that $g_t$ is also smooth at the critical points of $f_t$ because $g_t$ equals to the round metric in $\sph^2\setminus W'$ which is where the critical points of $f_t$ reside).
	
	It is easy to see that $g_t$ varies smoothly on $t$, and at $t=0$ equals to the round metric. Therefore this is an infinitesimal perturbation of the round metric where $f_t$ is an eigenfunction with eigenvalue $\lambda$ of $-\Delta_{g_t}$, and that eigenfunction has the specified nodal line configuration.
\end{proof}

\begin{proof}[theorem \ref{thm:main-perturb-result}]
	Let $X$ be an oval configuration which can be drawn in the $n\times n$ grid graph. By lemma \ref{lemma:from-drawing-to-coarse-perturb} there is a perturbation of $G'$ which nicely contains $X$. By lemma \ref{lemma:reduction} it follows that there is a perturbation of $G'$ which is equivalent to $X$. Therefore by lemma \ref{lemma:metric-perturb} the result follows.
\end{proof}

\appendix
\section{Appendix - Proof of lemma \ref{lemma:pointwise}}
\label{section:appendix-proof-of-pointwise-lemma}
\begin{proof}[lemma \ref{lemma:pointwise}]
	Working in a local coordinate chart, let $\mu = F d\nu$ where $\nu$ is the Lebesgue measure and $F>0$ is smooth. We need to find a symmetric matrix field $g$ such that $\omega = gu$ and $\left|g\right| = F^2$. Setting $A = \frac gF$, we get that the lemma is equivalent to the following claim:
	\begin{claim}\label{claim:pointwise-coords}
		Given two vectors $u,v\in\RR^2$ such that $\left\langle u,v\right\rangle>0$, there is a unique matrix $A\in SL_2(\RR)$ which is symmetric, positive-definite, and satisfies $Au=v$. Additionally, $A$ depends smoothly on $u,v$.
	\end{claim}
	Proving uniqueness:
	Let $J=\begin{pmatrix}
		0 & 1 \\ -1 & 0
	\end{pmatrix}$. Then we have $0 = \left\langle Jv,v\right\rangle=\left\langle Jv, Au\right\rangle = \left\langle AJv, u\right\rangle$. Therefore we have $AJv = \lambda Ju$ for some $\lambda\in\RR$. Note that $u, Jv$ are independent because we have $\left\langle u,v\right\rangle\neq 0$. Therefore we have
	\[
	A \begin{pmatrix} u & Jv\end{pmatrix} = \begin{pmatrix} v & \lambda Ju\end{pmatrix}
	.
	\]
	Taking determinants, we have
	\begin{multline*}
		-\left\langle u,v\right\rangle
		=
		-\left\langle u,v\right\rangle \det A
		=
		\left\langle u,J^2v\right\rangle \det A
		=
		\det A \det\left(\begin{pmatrix} u & Jv\end{pmatrix}\right)
		=\\=
		\det\left(\begin{pmatrix} v & \lambda Ju\end{pmatrix}\right) = \left\langle v,\lambda J^2 u\right\rangle = -\lambda \left\langle v,u\right\rangle
		,
	\end{multline*}
	therefore $\lambda=1$, and
	\[
	A = \begin{pmatrix} v &  Ju\end{pmatrix} \begin{pmatrix} u & Jv\end{pmatrix} ^{-1}
	=\frac1{\left\langle u,v\right\rangle} \begin{pmatrix} v &  Ju\end{pmatrix}  \begin{pmatrix} v &  Ju\end{pmatrix}^t
	,
	\]
	which means that $A$ is unique and smoothly depends on $u,v$. Existence follows by checking that the above formula for $A$ satisfies the conditions. (Note that the positive-definiteness of $A$ follows from the assumption that $\left\langle u,v\right\rangle>0$)
\end{proof}

\bibliographystyle{plain}
\if aa

\fi
\end{document}